\documentclass[reqno,english]{amsart}

\usepackage{amsmath,amssymb,verbatim,amsthm,mathrsfs,mathtools,hyperref}
\usepackage[capitalize]{cleveref}
\usepackage{tikz}
\usetikzlibrary{calc}
\usepackage{caption}

\newtheorem{thm}{Theorem}[section]
\newtheorem{prop}[thm]{Proposition}
\newtheorem{lem}[thm]{Lemma}

\newtheorem{fact}[thm]{Fact}

\newtheorem{quest}[thm]{Question}
\theoremstyle{definition}
\newtheorem{defn}[thm]{Definition}
 
\theoremstyle{remark}

\newcommand{\e}{\varepsilon}

\newcommand{\Rb}{\mathbb{R}}

\newcommand{\Nb}{\mathbb{N}}
\newcommand{\Zb}{\mathbb{Z}}

\newcommand{\Lc}{\mathcal{L}}
\newcommand{\Uc}{\mathcal{U}}

\newcommand{\Ds}{\mathscr{D}}

\newcommand{\abar}{\bar{a}}
\newcommand{\xbar}{\bar{x}}

\DeclareMathOperator{\tp}{tp}

\let\int\relax
\DeclareMathOperator{\int}{int}

\newcommand{\AND}{\text{\textbf{\itshape\&}}}

\DeclareMathOperator{\Th}{Th}

\def\Ind{\setbox0=\hbox{$x$}\kern\wd0\hbox to 0pt{\hss$\mid$\hss}
  \lower.9\ht0\hbox to 0pt{\hss$\smile$\hss}\kern\wd0}

\def\Notind{\setbox0=\hbox{$x$}\kern\wd0\hbox to 0pt{\mathchardef
    \nn=12854\hss$\nn$\kern1.4\wd0\hss}\hbox to
  0pt{\hss$\mid$\hss}\lower.9\ht0 \hbox to 0pt{\hss$\smile$\hss}\kern\wd0}

\def\nind{\mathop{\mathpalette\Notind{}}}

\DeclareMathOperator{\acl}{acl}
\DeclareMathOperator{\cl}{cl}

\newcommand{\res}{{\upharpoonright}}

\def\X{\mathbb{X}}
\def\Y{\mathbb{Y}}
\def\Wb{\mathbb{W}}
\def\Db{\mathbb{D}}

\makeatletter
\renewcommand\paragraph{\@startsection{paragraph}{4}{\z@}% 
  {2ex \@plus1ex \@minus.2ex}% 
  {-1em}% 
  {\normalfont\normalsize\textsc}} % 
\renewcommand\subparagraph{\@startsection{subparagraph}{5}{\parindent}% 
  {0ex \@plus1ex \@minus .2ex}% 
  {-1em}% 
  {\normalfont\normalsize\itshape}}
\makeatother

\newcommand{\To}{\Rightarrow}

\newcommand{\dinf}{d^{\inf}}

\begin{document}

\title{Some semilattices of definable sets in continuous logic}
\address{Department of Mathematics\\
  University of Maryland\\
  College Park, MD 20742, USA}
\author{James Hanson}
\email{jhanson9@umd.edu}
\date{\today}


\begin{abstract}
  In continuous first-order logic, the union of definable sets is definable but generally the intersection is not. This means that in any continuous theory, the collection of $\varnothing$-definable sets in one variable forms a join-semilattice under inclusion that may fail to be a lattice. We investigate the question of which semilattices arise as the collection of definable sets in a continuous theory. We show that for any non-trivial finite semilattice $L$ (or, equivalently, any finite lattice $L$), there is a superstable theory $T$ whose semilattice of definable sets is $L$. We then extend this construction to some infinite semilattices. In particular, we show that the following semilattices arise in continuous theories: $\alpha+1$ and $(\alpha+1)^\ast$ for any ordinal $\alpha$, a semilattice containing an exact pair above $\omega$, and the lattice of filters in $L$ for any countable meet-semilattice $L$. By previous work of the author, this establishes that these semilattices arise in stable theories \cite{2021arXiv210613261H}. The first two are done in languages of cardinality $\aleph_0 + |\alpha|$, and the latter two in countable languages.
\end{abstract}

\maketitle

\section*{Introduction}
\label{sec:intro}

Continuous first-order logic, introduced in its modern form in \cite{MTFMS}, is a generalization of ordinary first-order logic that deals with structures comprising complete metric spaces and uniformly continuous predicates and functions, called \emph{metric structures}. % 

In a metric structure $M$, a closed set $D \subseteq M$ is \emph{definable} if its distance predicate $\inf\{d(x,y) : y \in D\}$ is equivalent to a formula. (We take formulas to be closed under uniformly convergent limits.) These are precisely the sets that admit relative quantification in the sense that for any formula $\varphi(x,y)$, there is a formula equivalent to $\inf_{y \in D}\varphi(x,y)$. We will conflate a definable set $D \subseteq M$ with the corresponding closed set of types $\{\tp(a) : a \in D^N,~N \succeq M\}$ in $S_1(T)$, and we will abuse terminology by referring to such sets as \emph{definable} as well.

While definable sets are useful things to have, they are not always plentiful. There are non-trivial metric structures $M$ in which the only definable sets (without parameters) are $\varnothing$ and $M$. Similarly, while definable sets in discrete logic enjoy the structure of a Boolean algebra, only the union of two definable sets is guaranteed to be definable in continuous logic. The complement of a definable set is generally not even closed and the intersection of two definable sets may fail to be definable. This means that the partial order of definable subsets of a metric structure is generally only a bounded join-semilattice. In situations like this, logicians typically find irresistible the question of which semilattices arise in this manner. We will address some aspects of this question in this paper.

The easiest restriction to establish is on cardinality. The collection of definable subsets of $M$ is always closed in the Hausdorff metric. This means that it can only ever have the cardinality of a complete metric space. Basic set-theoretic topology then establishes that if $M$ is a metric structure whose language $\Lc$ has cardinality $\kappa$, then the collection $\Ds$ of definable subsets of $M$ has either $|\mathscr{D}| \leq \kappa$ or $|\mathscr{D}| = \kappa^{\aleph_0}$. If $\Lc$ is countable, then $\mathscr{D}$ is a Polish space and $(D,E) \mapsto D\cup E$ is a Borel function, whereby descriptive set theory comes to bear. That said, the author is not aware of any results regarding Polish semilattices that aren't special cases of results regarding Polish partial orders. 

A more specialized question one might ask is whether various model-theoretic properties impose restrictions on the semilattices of definable sets that might appear. Aside from discreteness itself, only a handful of such restrictions are known. In his thesis \cite[Ch.\ 2.4, 2.5, 5.5]{HansonThesis}, the author studied an abundance condition for definable sets and established conditions under which it occurs, which he referred to as \emph{dictionaricness}. While this condition seems to have been identified folklorically, at the time there were no published systematic studies of it. A type space $S_{\xbar}(A)$ is \emph{dictionaric} if it has a basis of definable neighborhoods (where a neighborhood is not required to be open). A theory is dictionaric if all of its type over any set of parameters are dictionaric. There are only three broad classes of theories that are known to be dictionaric: theories whose type spaces are all totally disconnected,\footnote{This occurs if and only if the theory is bi-interpretable with a (possibly many-sorted) discrete theory.} theories for which $S_n(\abar)$ is CB-analyzable for every $n < \omega$ and finite tuple $\abar$, and randomizations (of discrete or continuous theories, regardless of whether the original theory is dictionaric). The second class of theories includes $\omega$-stable theories, implying that all $\omega$-stable theories are dictionaric. % 

On the other hand, the author showed in \cite[Thm.\ 5.2]{2021arXiv210613261H} that every compact topometric space $(X,\tau,d)$ with $d$ open\footnote{$d$ is open if $U^{<r}\coloneqq \bigcup_{x \in U}B_{<r}(x)$ is open for every open $U$ and $r > 0$.} is isometrically homeomorphic (i.e., isomorphic in the sense of topometric spaces) to $S_1(T)$ for a strictly stable theory $T$. This means that stability alone imposes no restrictions on the semilattice of definable sets.

This leaves a gap though. $\omega$-stability imposes radical restrictions. Mere stability imposes none. What, if anything, does superstability impose? Many of the explicitly constructed counterexamples in \cite[Sec.\ C.1]{HansonThesis} are superstable, which bodes ill for the prospect that superstability entails anything in this regard. For instance, there is a weakly minimal theory $T$ with trivial geometry for which $S_1(T)$ is homeomorphic to $[0,1]$ but has only $\varnothing$ and $S_1(T)$ as definable sets.\footnote{$T$ can be taken to be the theory of $\Nb$ with a discrete metric and $\cos(x)$ as a predicate.} Nevertheless, it is entirely possible that there are subtle restrictions on the class of definable sets entailed by superstability.

In this paper we will give evidence that superstability does not entail any such restrictions. We will show that for every non-trivial finite semilattice $L$, there is a weakly minimal theory $T$ with trivial geometry such that the semilattice of definable subsets of $S_1(T)$ is isomorphic to $L$. Note that since finite semilattices are always complete, they are always lattices, so our result shows that the finite semilattices of definable sets are precisely the finite lattices.\footnote{If $T$ is the inconsistent theory, then a pedantic reading of definitions gives that $S_1(T)$ is $\varnothing$, which has the trivial lattice as its semilattice of definable sets.} This does mean that our result is limited in the sense that superstability might impose restrictions on the semilattice of definable sets that are weaker than the restrictions already imposed by finiteness.

\section{Finite semilattices of definable sets}
\label{sec:building-fin-lat}

The following facts will be useful to keep in mind during the construction. Topological operators such as the interior, $\int A$, are always computed in the compact logic topology. We take superscript operators to bind more tightly than prefix and infix operators, so $\int A^{<\e}$ is $\int(A^{<\e})$,  $\cl(A\cap B)^{<\e}$ is $\cl((A\cap B)^{<\e})$, and $A\cap B^{<\e}$ is $A \cap (B^{<\e})$.

\begin{fact}\label{fact:def-set-basic}
  For any type space $S_1(T)$, a closed set $D \subseteq S_1(T)$ is definable if and only if $D \subseteq \int D^{<r}$ for every $r > 0$.

  For any topometric space $X$ and set $Q \subseteq X$, the following are equivalent.
  \begin{enumerate}
  \item $Q \subseteq \int Q^{<r}$ for every $r > 0$.
  \item $Q \subseteq \int Q^{<r}$ for arbitrarily small $r > 0$.
  \end{enumerate}
  If the metric on $X$ is open, then 1 and 2 are also equivalent to 3.
  \begin{enumerate}
\setcounter{enumi}{2}
  \item $Q^{<r}$ is open for every $r > 0$.
  \end{enumerate}
\end{fact}
\begin{proof}
  The statement regarding definable sets is equivalent to \cite[Prop.\ 9.19]{MTFMS}.
  
  1 clearly implies 2. Assuming 2, then for any $r > 0$, we can find positive $s < r$ such that $Q \subseteq \int Q^{<s}$, but $\int Q^{<s} \subseteq Q^{<s} \subseteq Q^{<r}$, whence $Q \subseteq \int Q^{<r}$.

  3 clearly implies 1 and 2, so assume the metric is open. For any $x \in Q^{<r}$, let $s = d(x,Q) < r$. We now have that $x \in Q^{<s} \subseteq (\int Q^{<r-s})^{<s} \subseteq Q^{<r}$ by the triangle inequality. Since $(\int Q^{<r-s})^{<s}$ is open, $x \in \int Q^{<r}$. Since we can do this for every $x \in Q^{<r}$, $Q^{<r}$ is open.
\end{proof}

Given \cref{fact:def-set-basic}, we will use the following definition.
\begin{defn}\label{defn:def-set-topo}
  In any topometric space $(X,\tau,d)$, a closed set $D \subseteq X$ is \emph{definable} if $D \subseteq \int D^{<r}$ for every $r > 0$.
\end{defn}

In \cref{sec:some-infin-lat}, we will use \cref{defn:def-set-topo} even when $X$ is not compact.

Perhaps generalizing the term `definable set' to arbitrary topometric spaces like this is ill advised, but at the moment it doesn't seem that there are any applications of \cref{defn:def-set-topo} outside of the context of type spaces in continuous logic.

\subsection{Circuitry}
\label{sec:circuitry}

We will start our construction by building the type space $S_1(T)$ as an explicit topometric space. We will then argue that what we have built actually is $S_1(T)$ for some weakly minimal $T$ with trivial geometry. The construction proceeds by building something reminiscent of a logical circuit consisting of `wires' and `gates.' Unfortunately the metaphor is somewhat backwards in that it will make sense to regard a wire as `on' if it is \emph{disjoint} from the definable set in question. 

\begin{defn}\label{defn:crisp-embeddedness}
  Given a topometric space $(X,\tau,d)$, a set $A \subseteq X$ is \emph{crisply embedded in $X$} if $d(a,x) = 1$ for any $a \in A$ and $x \in X \setminus A$. If $\{x\}$ is crisply embedded in $X$, we may also say that the point $x$ is \emph{crisply embedded}.

  A point $a \in X$ is \emph{metrically $\e$-isolated} if $d(a,x) \geq \e$ for any $x \in X \setminus \{a\}$.
\end{defn}

Throughout the paper all metrics will be $[0,1]$-valued. Note that $a$ is crisply embedded if and only if it is metrically $1$-isolated.

\begin{defn}\label{defn:soldering}
  Given topometric spaces $X$ and $Y$, the \emph{coproduct of $X$ and $Y$}, written $X\oplus Y$ is the topometric space with underlying topological space $X \sqcup Y$ where the metric is extended so that $d(x,y) = 1$ for any $x \in X$ and $y \in Y$.

  Given a topometric space $X$ and two crisply embedded points $x$ and $y$, the topometric space produced by \emph{soldering $x$ and $y$ together} is the topometric space whose underlying topological space is $X$ with $x$ and $y$ topologically glued and in which the metric is defined so that $d(z,w)$ is unchanged for any $z$ and $w$ in $X$ with $\{z,w\} \neq \{x,y\}$. Given a finite set of crisply embedded points $X_0 \subseteq X$, we define \emph{soldering together the points of $X_0$} similarly.

  Given two topometric spaces $X$ and $Y$ with crisply embedded $x \in X$ and $y \in Y$, the topometric space produced by \emph{soldering $x$ and $y$ together} is the topometric space produced by soldering $x$ and $y$ together in $X \oplus Y$.
\end{defn}

It is easy to verify that the objects described in \cref{defn:soldering} are in fact topometric spaces.

A fact that we will be frequently use implicitly is this: If $X$ is a metric space and $D,E \subseteq X$, then for any $r > 0$, $(D\cap E)^{<r}\cap E = \bigcup_{x \in D\cap E}B_{<r}^{E}(x)$, where $B_{<r}^E(x)$ is the open ball of radius $r$ around $x$ in the metric space $(E,d)$. In other words, $(D\cap E)^{<r}\cap E$ is $(D\cap E)^{<r}$ `computed in $E$.'  

\begin{lem}\label{lem:soldering}
Let $X$ and $Y$ be topometric spaces. 
\begin{enumerate}
\item $D \subseteq X \oplus Y$ is definable if and only if $D \cap X$ is definable in $X$ and $D\cap Y$ is definable in $Y$.
\item $X \oplus Y$ has an open metric if and only if $X$ and $Y$ have open metrics.
\end{enumerate}

Let  $Z$ be topometric spaces, let $z_0$ and $z_1$ be crisply embedded points in $Z$, let $W$ be $Z$ with $z_0$ and $z_1$ soldered together, let $w \in W$ be the point corresponding to $z_0$ and $z_1$, and let $\pi : Z \to W$ be the quotient map.
\begin{enumerate}
  \setcounter{enumi}{2}
\item For any closed $D \subseteq Z$, $\pi[D] \subseteq W$ is definable if and only if either
  \begin{itemize}
  \item $\{z_0,z_1\}\cap D = \varnothing$ and $D$ is definable or
  \item $\{z_0,z_1\}\cap D \neq \varnothing$ and $D \cup \{z_0,z_1\}$ is definable.
  \end{itemize}
\item $W$ has an open metric if and only if $Z$ has an open metric.
\end{enumerate}

\end{lem}
\begin{proof}
  1 and 2 follow immediately from the fact that for any positive $r \leq 1$ and any $A \subseteq X \oplus Y$, $A^{<r} = (X\cap A)^{<r} \cup (Y\cap A)^{<r}$.

  For 3, suppose $\pi[D]$ is definable. If $w \notin \pi[D]$, then we have that $\pi[D]$ is definable as a subset of $W\setminus \{w\}$, which is open. Since $\pi \res (Z\setminus \{z_0,z_1\})$ is an isometric homeomorphism, this is enough to establish that $D$ is definable in $Z \setminus \{z_0,z_1\}$ and therefore also in $Z$. Every step in this argument is reversible, so we also have that if $\{z_0,z_1\}\cap D = \varnothing$ and $D$ is definable, then $\pi[D]$ is definable.

  If $w \in \pi[D]$, then $\pi^{-1}[D] = D \cup \{z_0,z_1\}$. For every $r > 0$, we now have that $\pi^{-1}[D^{<r}] = (D\cup \{z_0,z_1\})^{<r}$. Therefore $D \cup \{z_0,z_1\}$ is definable. Again, this argument is reversible, so we have that if $D\cup\{z_0,z_1\}$ is definable, then $\pi[D\cup \{z_0,z_1\}]$ is definable.

  For 4, if $Z$ has an open metric, then for any $U \subseteq W$ and any $r > 0$, we clearly have that $U^{<r} = \pi[\pi^{-1}[U]^{<r}]$ is open. On the other hand, if $W$ has an open metric, then for any $U \subseteq Z\setminus \{z_0,z_1\}$ and any $r > 0$, we have that $U^{<r} = \pi^{-1}[\pi[U]^{<r}]$. If $U$ contains one of $z_0$ and $z_1$, then for any positive $r \leq 1$, we have that $U^{<r} = (U\cap \{z_0,z_1\}) \cup \pi^{-1}[\pi[U\setminus \{z_0,z_1\}]^{<r}]$. Since $U$ is a neighborhood of the points in $U \cap \{z_0,z_1\}$, this set is always open. Furthermore, for any $r > 1$, we have $U^{<r} = Z$. Therefore $Z$ has an open metric.
\end{proof}

Note that \cref{lem:soldering} implies that if $x$ is crisply embedded in $X$ and $y$ is crisply embedded in $Y$ and $W$ is the result of soldering $x$ and $y$ together, then
\begin{itemize}
\item for any closed $D \subseteq W$, $D$ is definable if and only if $D \cap X$ is definable in $X$ and $D \cap Y$ is definable in $Y$ and
\item $W$ has an open metric if and only if $X$ and $Y$ have open metrics.
\end{itemize}

We may write ordered tuples of numbers with angle brackets to avoid confusion with open intervals.

\begin{defn}
  The \emph{AND gate space} is the topometric space $(\AND,\tau,d)$ where $\AND$ is the subset of $\Rb^2$ given by
  \[
    ([-2,1]\times \{-1,1\}) \cup \{\langle x,\pm x \rangle : -1 \leq x \leq 1\} \cup ([0,2]\times \{0\})
  \]
  (see \cref{fig:AND-gate}), $\tau$ is the subspace topology, and $d$ is the unique metric satisfying
  \begin{itemize}
  \item if $\langle x,y \rangle \in ([-2,1]\times \{-1,1\})\cup([1,2]\times \{0\})$ and $\langle z,w \rangle \neq \langle x,y \rangle$ (see \cref{fig:disc-met-region}), then $d(\langle x,y \rangle,\langle z,w \rangle) = 1$,
  \item if $x \neq z$, then $d(\langle x,y \rangle,\langle z,w \rangle) = 1$, and
  \item if $\langle x,y \rangle$ and $\langle z,w \rangle$ are both in the set $\{\langle x, \pm x \rangle : -1 \leq x \leq 1\} \cup ([0,1]\times \{0\})$, $x =z$, and $y \neq w$, then $d(\langle x,y \rangle,\langle z,w \rangle) = \max(|y|,|w|)$.
  \end{itemize}

  The points $\langle -2,1 \rangle$ and $\langle -2,-1 \rangle$ are the \emph{input vertices} of $\AND$, and the point $\langle 2,0 \rangle$ is the \emph{output vertex}.
\end{defn}
Note that the vertices of $\AND$ are crisply embedded.

\begin{figure}
  \captionsetup{width=0.5\textwidth}
  \centering
  \begin{minipage}{0.49\textwidth}
    \centering
    \begin{tikzpicture}
      \draw[line width=5pt,draw=none,line cap=round,rounded corners=1] (-2,1) -- (1,1);     
      \draw[line width=5pt,draw=none,line cap=round,rounded corners=1] (-2,-1) -- (1,-1);
      \draw[line width=5pt,draw=none,line cap=round,rounded corners=1] (1,0) -- (2,0);
      \draw[very thick,rounded corners=1] (-2,1) node[left]{In}-- (-1,1) -- (1,1) -- (-1,-1);
      \draw[very thick,rounded corners=1] (-2,-1) node[left]{In} -- (-1,-1) -- (1,-1) -- (-1,1);
      \draw[very thick,rounded corners=1] (0,0) -- (1,0) -- (2,0) node[right]{Out}; % 
    \end{tikzpicture}
    \caption{$\AND$, the AND gate space}
    \label{fig:AND-gate}
  \end{minipage}
  \begin{minipage}{0.49\textwidth}
    \centering
    \begin{tikzpicture}
      \draw[line width=5pt,opacity=0.5,line cap=round,rounded corners] (-2,1) -- (1,1);     
      \draw[line width=5pt,opacity=0.5,line cap=round,rounded corners=1] (-2,-1) -- (1,-1);
      \draw[line width=5pt,opacity=0.5,line cap=round,rounded corners=1] (1,0) -- (2,0);
      \draw[very thick,rounded corners=1] (-2,1)  node[left]{\phantom{In}}-- (-1,1)  -- (1,1) -- (-1,-1); % 
      \draw[very thick,rounded corners=1] (-2,-1)  node[left]{\phantom{In}} -- (-1,-1) -- (1,-1) -- (-1,1); % 
      \draw[very thick,rounded corners=1] (0,0) -- (1,0) -- (2,0) node[right]{\phantom{Out}}; % 
    \end{tikzpicture}
    \caption{Region with discrete metric}
    \label{fig:disc-met-region}
  \end{minipage}
\end{figure}

\begin{lem}\label{lem:wires}
  Fix a topometric space $X$. For any $\e > 0$, if $U \subseteq X$ is a connected open set such that every $x \in U$ is metrically $\e$-isolated, then for any definable set $D \subseteq X$, either $U \subseteq D$ or $U \cap D = \varnothing$.
\end{lem}
\begin{proof}
  Fix a connected open set $U$ and an $\e > 0$ such that every $x \in U$ is metrically $\e$-isolated. Fix a definable set $D$. We have that $U \cap D^{<\frac{1}{2}\e} = U \cap D$, so
  \[
    U \cap D \subseteq U \cap \int D^{<\frac{1}{2}\e} \subseteq U \cap D^{<\frac{1}{2}\e} = U \cap D
  \]
  and $U \cap D$ is open and therefore relatively clopen in $U$. Therefore either $U \cap D = \varnothing$ or $U \cap D = U$, as required. 
\end{proof}

The name of the AND gate space is justified by the following proposition.

\begin{prop}\label{prop:AND-function}
  The only non-empty definable proper subsets of $\AND$ are
  \begin{enumerate}
  \item \label{top-lobe}$\{\langle x,y \rangle \in \AND: x > 0\} \cup \{\langle 0,0 \rangle\}$,
  \item \label{bottom-lobe}$\{\langle x,y \rangle \in \AND: x < 0\} \cup \{\langle 0,0 \rangle\}$,
  \item the union of \ref{top-lobe} and \ref{bottom-lobe}, 
  \item the union of \ref{top-lobe} and $(0,2]\times \{0\}$, and
  \item the union of \ref{bottom-lobe} and $(0,2]\times \{0\}$.
  \end{enumerate}

  In particular, a definable subset $D\subseteq \AND$ is uniquely determined by its intersection with $\{\langle -2,1 \rangle, \langle -2,-1 \rangle,\langle 2,0 \rangle\}$ and the only restriction is that if $\langle -2,1 \rangle \notin D$ and $\langle -2,-1 \rangle \notin D$, then $\langle 2,0 \rangle \notin D$.
\end{prop}
\begin{proof}
  In order to verify that the listed sets are definable, it is sufficient by symmetry and unions to check it only for 1 and 4. Let $D_1$ be the set described in 1. For any positive $r \leq 1$, we have that $D^{<r}_1$ is $D_1 \cup \{\langle x,-|x| \rangle: |x| < r\} \cup \{\langle x,0 \rangle : x < r \}$, which is an open set. Therefore $D_1$ is definable. If $D_4$ is the set described in 4, then $D^{<r}_4$ is $D^{<r}_1 \cup D_4$, which is also open. Therefore $D_4$ is definable.

  Now suppose that $D \subseteq \AND$ is definable. $\AND$ is a topological graph. By an \emph{edge} of $\AND$, we mean a maximal open subset homeomorphic to $(0,1)$. There are $9$ edges in $\AND$ corresponding to the graph theoretic edges. Each edge $U$ of $\AND$ can be written as a union $\bigcup_{n<\omega} U_n$ of connected open sets such that any $x \in U_n$ is metrically $2^{-n-1}$-isolated. Therefore by \cref{lem:wires} and since each $U$ is connected, we have that either $U \subseteq D$ or $U \cap D = \varnothing$.

  Since $D$ is closed, it must contain the closure of any edge it contains. Now suppose that $D$ contains one of the four points $\langle \pm 1,\pm 1 \rangle$. Call this point $\langle x,y \rangle$ Suppose that $\langle x,y \rangle \notin \int D$. We can then find an $r > 0$ small enough that for some open neighborhood $V \ni \langle x,y \rangle$, $D^{<r}\cap V = D\cap V$, which again is a contradiction.

  What we have established now is enough to show that $D$ must be a (possibly empty) union of $D_1$, $D_2$ (the definable set in 2), and $[0,2]\times \{0\}$. If $D$ is $\varnothing$ or $\AND$ or is on the list in the statement of the proposition, then we are done. The only other possibility is that $D = [0,2]\times \{0\}$, but if this is the case then $D^{<\frac{1}{2}} = D \cup \{\langle z,\pm z \rangle: 0 \leq z < \frac{1}{2}\}$, which is not open. So $D$ cannot be this set and we are done. The `In particular' statement follows immediately.
\end{proof}

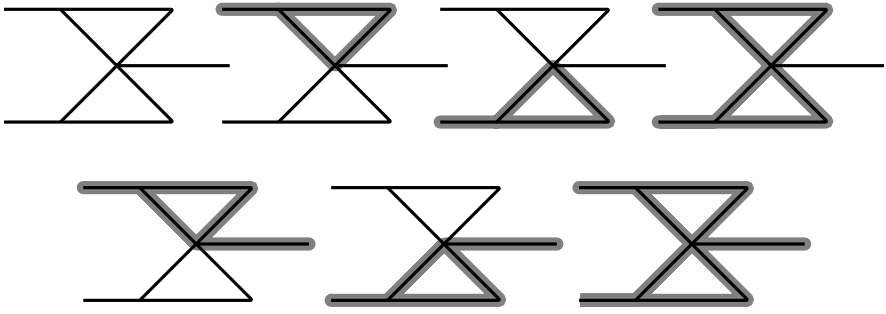
\begin{figure}
  \centering
  \begin{tikzpicture}[scale=0.75]
    \begin{scope}[xshift=-165,yshift=0]
      \draw[very thick,rounded corners=1] (-2,1) -- (-1,1) -- (1,1) -- (-1,-1);
      \draw[very thick,rounded corners=1] (-2,-1) -- (-1,-1) -- (1,-1) -- (-1,1);
      \draw[very thick,rounded corners=1] (0,0) -- (1,0) -- (2,0); % 
    \end{scope}
    \begin{scope}[xshift=-55,yshift=0]
      \draw[line width=5pt,opacity=0.5,line cap=round,rounded corners=1] (-2,1) -- (1,1) -- (0,0) -- (-1,1);     
      \draw[very thick,rounded corners=1] (-2,1) -- (-1,1) -- (1,1) -- (-1,-1);
      \draw[very thick,rounded corners=1] (-2,-1) -- (-1,-1) -- (1,-1) -- (-1,1);
      \draw[very thick,rounded corners=1] (0,0) -- (1,0) -- (2,0); % 
    \end{scope}
    \begin{scope}[xshift=55,yshift=0]
      \draw[line width=5pt,opacity=0.5,line cap=round,rounded corners=1] (-2,-1) -- (1,-1) -- (0,0) -- (-1,-1);     
      \draw[line width=5pt,opacity=0,line cap=round,rounded corners=1] (-2,1) -- (1,1);     
      \draw[line width=5pt,opacity=0,line cap=round,rounded corners=1] (-2,-1) -- (1,-1);
      \draw[line width=5pt,opacity=0,line cap=round,rounded corners=1] (1,0) -- (2,0);
      \draw[very thick,rounded corners=1] (-2,1) -- (-1,1) -- (1,1) -- (-1,-1);
      \draw[very thick,rounded corners=1] (-2,-1) -- (-1,-1) -- (1,-1) -- (-1,1);
      \draw[very thick,rounded corners=1] (0,0) -- (1,0) -- (2,0); % 
    \end{scope}
    \begin{scope}[xshift=165,yshift=0]
      \draw[line width=5pt,opacity=0.5,line cap=round,rounded corners=1] (-2,1) -- (1,1) -- (0,0) -- (-1,1) -- (0,0) -- (1,-1) -- (-2,-1) -- (-1,-1) -- (0,0) -- (-1,-1) -- (-2,-1);     
      \draw[very thick,rounded corners=1] (-2,1) -- (-1,1) -- (1,1) -- (-1,-1);
      \draw[very thick,rounded corners=1] (-2,-1) -- (-1,-1) -- (1,-1) -- (-1,1);
      \draw[very thick,rounded corners=1] (0,0) -- (1,0) -- (2,0); % 
    \end{scope}
    \begin{scope}[xshift=-125,yshift=-90]
      \draw[line width=5pt,opacity=0.5,line cap=round,rounded corners=1] (-2,1) -- (1,1) -- (0,0) -- (-1,1) -- (0,0) -- (2,0);     
      \draw[very thick,rounded corners=1] (-2,1) -- (-1,1) -- (1,1) -- (-1,-1);
      \draw[very thick,rounded corners=1] (-2,-1) -- (-1,-1) -- (1,-1) -- (-1,1);
      \draw[very thick,rounded corners=1] (0,0) -- (1,0) -- (2,0); % 
    \end{scope}
    \begin{scope}[xshift=0,yshift=-90]
      \draw[line width=5pt,opacity=0.5,line cap=round,rounded corners=1] (-2,-1) -- (1,-1) -- (0,0) -- (-1,-1) -- (0,0) -- (2,0);     
      \draw[line width=5pt,opacity=0,line cap=round,rounded corners=1] (-2,1) -- (1,1);     
      \draw[line width=5pt,opacity=0,line cap=round,rounded corners=1] (-2,-1) -- (1,-1);
      \draw[line width=5pt,opacity=0,line cap=round,rounded corners=1] (1,0) -- (2,0);
      \draw[very thick,rounded corners=1] (-2,1) -- (-1,1) -- (1,1) -- (-1,-1);
      \draw[very thick,rounded corners=1] (-2,-1) -- (-1,-1) -- (1,-1) -- (-1,1);
      \draw[very thick,rounded corners=1] (0,0) -- (1,0) -- (2,0); % 
    \end{scope}
    \begin{scope}[xshift=125,yshift=-90]
      \draw[line width=5pt,opacity=0.5,line cap=round,rounded corners=1] (-2,1) -- (1,1) -- (0,0) -- (-1,1) -- (0,0) -- (1,-1) -- (-2,-1) -- (-1,-1) -- (0,0) -- (-1,-1) -- (-2,-1) -- (-1,-1) -- (0,0) -- (2,0);     
      \draw[very thick,rounded corners=1] (-2,1) -- (-1,1) -- (1,1) -- (-1,-1);
      \draw[very thick,rounded corners=1] (-2,-1) -- (-1,-1) -- (1,-1) -- (-1,1);
      \draw[very thick,rounded corners=1] (0,0) -- (1,0) -- (2,0); % 
    \end{scope}
  \end{tikzpicture}
  \caption{The definable subsets of $\AND$} % 
  \label{fig:non-triv-def-sets}
\end{figure}

So we see that $\AND$ functions like an AND gate
in the following sense: Given a definable set $D \subseteq \AND$, we think of a vertex as being `on' if it is not contained in the definable set in question. We then have by \cref{prop:AND-function} that if the input vertices are on, the output vertex must be on as well, but there are no other restrictions on the configuration of the gate.

Strictly speaking, real-world AND gates usually don't have specified behavior for states analogous to 1, 2, or 3 in \cref{prop:AND-function}, since normally the output is meant to be thought of as a function of the inputs. This means that there are two ways to interpret what configurations should be possible. Here we have interpreted the operation of an AND gate in an `if, then' manner, where the output is on if the inputs are both on. The other way would be to interpret it in an `if and only if' manner, where the output is on if and only if the inputs are both on. It is actually easier to build a topometric space that accomplishes this, but in the end we would need to implement something that behaves like $\AND$ and the resulting construction is ultimately more complicated.

\subsection{Building type spaces}
\label{sec:enc-fin-lat-typ-spa}

Now we will use $\AND$ to build type spaces with arbitrary finite lattices as their semilattices of definable sets.

\begin{defn}\label{defn:lattice-space}
  For any bounded lattice $L$, write $L^-$ for the set $L \setminus \{1^L\}$ (where $1^L$ is the top element of $L$).

  For a finite lattice $L$, we write $\X(L)$ for the topometric space constructed in the following manner: For each triple $(a,b,c) \in (L^-)^3$ satisfying $a \wedge b \leq c$, take a copy of $\AND$ with the two input vertices labeled $a$ and $b$ and the output vertex labeled $c$. For each $a \in L^-$, solder together all vertices labeled $a$. $\X(L)$ is the resulting space. We write $x_a$ for the point in $\X(L)$ corresponding to the vertices labeled $a$, and we write $N(L)$ for $\{x_a : a \in L^-\}$.
\end{defn}

Note that \cref{defn:lattice-space} will include many unnecessary copies of $\AND$, such as copies corresponding to triples of the form $(a,a,a)$. It is only really necessary to include copies corresponding to some presentation of the lattice. $N_5^-$, for instance, has $42$ ordered triples $( a,b,c )$ satisfying $a\wedge b \leq c$, but only $4$ of these are needed to produce a topometric space whose semilattice of definable sets is isomorphic to $N_5$, as depicted in \cref{fig:N5-space}.  

This example also establishes that even if a type space is `planar' (as in, embeddable in $\Rb^2$) and has finitely many definable sets, the lattice of definable sets might fail to be modular.

\begin{figure}
  \centering
  \begin{tikzpicture}
    \begin{scope}[shift={(-3.75,-2.25)}]
      \draw[very thick] (0,0) node[circle,fill=black,scale=0.6]{} -- (1.75,1.25)  node[circle,fill=black,scale=0.6]{} -- (1.75,3.25)  node[circle,fill=black,scale=0.6]{} -- (0,4.5)  node[circle,fill=black,scale=0.6]{} -- (-2,2.25)  node[circle,fill=black,scale=0.6]{} -- cycle;
    \end{scope}
    \draw[very thick,-stealth] (-0.9,0) -- (0.75,0);

    \begin{scope}[shift={(3.75,-2.25)}]
      \coordinate (root) at (0,0);
      \coordinate (LRN) at (1.75,1.25); % 
      \coordinate (URN) at (1.75,3.25); % 
      \coordinate (LN) at (-2,2.25); % 
      
      \draw[draw=none] (LN) -- (0,0) coordinate[pos=0.42] (LL) -- (LRN) coordinate[pos=0.6] (LR) -- (URN) coordinate[pos=0.55] (UR);

      \draw[very thick] (LL) -- (LN) node[circle,fill=black,scale=0.6]{};

      \draw[very thick] (LR) -- (LRN) node[circle,fill=black,scale=0.6]{};

      \draw[very thick] (UR) -- (URN) node[circle,fill=black,scale=0.6]{};

      \draw (0,0) node[circle,fill=black,scale=0.6]{};

      \coordinate (mid) at ($ (UR) + (-1.75,0.6)$);
      \draw[very thick] (mid) -- (0,0);

      \begin{scope}[shift={(LR)},rotate=35.54]
        \draw[very thick] (-0.3,0.3) -- (0,0) -- (0.3,-0.3) -- (-0.3,-0.3) -- (0.3,0.3) -- cycle;
        \draw[very thick] (-0.3,0.3) .. controls ++(-0.5,0) and ++(0,0) .. (root);
        \draw[very thick] (-0.3,-0.3) .. controls ++(-0.5,0) and ++(0,0) .. (root);
      \end{scope}

      \begin{scope}[shift={(UR)},rotate=90]
        \draw[very thick] (-0.3,0.3) -- (0,0) -- (0.3,-0.3) -- (-0.3,-0.3) -- (0.3,0.3) -- cycle;
        \draw[very thick] (-0.3,0.3) .. controls ++(-0.4,0) and ++(0,0) .. (LRN);
        \draw[very thick] (-0.3,-0.3) .. controls ++(-0.4,0) and ++(0,0) .. (LRN);
      \end{scope}

      \begin{scope}[shift={(LL)},rotate=131.6]
        \draw[very thick] (-0.3,0.3) -- (0,0) -- (0.3,-0.3) -- (-0.3,-0.3) -- (0.3,0.3) -- cycle;
        \draw[very thick] (-0.3,0.3) .. controls ++(-0.5,0) and ++(0,0) .. (root);
        \draw[very thick] (-0.3,-0.3) .. controls ++(-0.5,0) and ++(0,0) .. (root);
      \end{scope}

      \begin{scope}[shift={(mid)},rotate=-90]
        \draw[very thick] (-0.3,0.3) -- (0,0) -- (0.3,-0.3) -- (-0.3,-0.3) -- (0.3,0.3) -- cycle;
        \draw[very thick] (-0.3,0.3) .. controls ++(-1.5,0) and ($ (URN) + (-0.5,-0.1) $) .. (URN);
        \draw[very thick] (-0.3,-0.3) .. controls ++(-1.5,0) and ($ (LN) + (-1.5,0.25) $) .. (LN);
      \end{scope}
    \end{scope}
  \end{tikzpicture}
\caption{$\X(N_5)$ (with unnecessary copies of $\AND$ removed)}
  \label{fig:N5-space}
\end{figure}
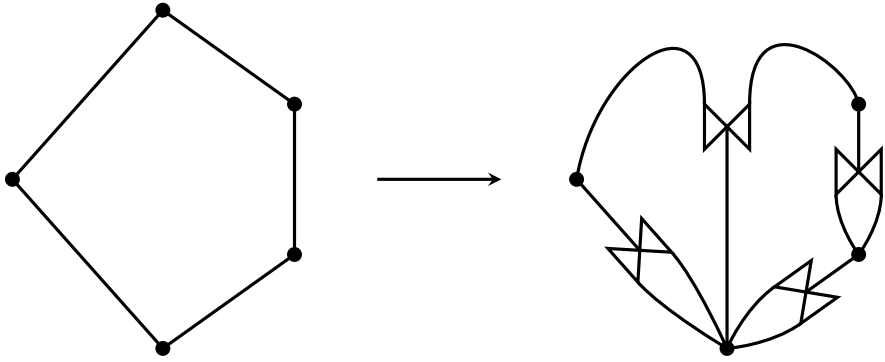

\begin{lem}\label{lem:lattice-circuit}
  Fix a finite lattice $L$ and a definable set $D$.
  \begin{enumerate}
  \item Any definable set $D \subseteq \X(L)$ is uniquely determined by $D \cap N(L)$.
  \item For any $a,b,c \in L^-$ with $a \wedge b \leq c$, if $a\notin D$ and $b \notin D$, then $c \notin D$. In particular, the set $\{a \in L^- : x_a \notin D\}\cup\{1^L\}$ is a filter of the lattice $L$.
  \item For any non-empty filter $F \subseteq L$, there is a definable set $D\subseteq \X(L)$ such that $F = \{a \in L^- : x_a \notin D\}\cup \{1^L\}$. 
  \end{enumerate}
\end{lem}
\begin{proof}
  1 and 2 follow from \cref{lem:soldering} and \cref{prop:AND-function}.

  For 3, given a non-empty filter $F$, let $D \subseteq \X(L)$ be the unique set satisfying $D \cap N(L) = \{x_a : a \in F^-\}$ such that for each copy $A$ of $\AND$ in $\X(L)$, $D\cap A$ is definable in $A$. This set is definable by \cref{lem:soldering} and \cref{prop:AND-function}.
\end{proof}

\begin{prop}\label{prop:XL-lattice-iso}
  For any finite lattice $L$, the semilattice of definable sets in $\X(L)$ is isomorphic to $L$.
\end{prop}
\begin{proof}
  For any $a \in L$, let $F_a = \{b \in L : b \geq a\}$ and let $D_a$ be the corresponding definable set in $\X(L)$. Clearly $D_a = D_b$ if and only if $a=b$, since $F_a^- = F_b^-$ if and only if $a=b$, so the map $a \mapsto D_a$ is an injection. Furthermore, $F_{a \vee b} = F_a \cap F_b$ and the definable set corresponding to $F_a \cap F_b$ is $D_a \cup D_b$, so the map $a \mapsto D_a$ preserves joins. $F_{1^L}$ is clearly $\X(L)$ and $F_{0^L}$ is clearly $\varnothing$. Therefore $a \mapsto D_a$ is a bounded semilattice homomorphism. Finally, since $L$ is finite, every non-empty filter is of the form $F_a$, so we have that $a\mapsto D_a$ is a surjection and hence a lattice isomorphism.
\end{proof}

Note that for each $a \in L$, $D_a$ is the unique maximal definable set \emph{not} containing $x_a$.\footnote{$D_{1^L} = \X(L)$ as $x_{1^L}$ does not exist and is therefore not contained in any definable set.} In \cref{fig:N5-space}, the left-hand element of $N_5$ maps to the definable set containing the two right-hand copies of $\AND$ and the right-hand side of the center copy of $\AND$, for instance.

\subsection{Weak minimality}
\label{sec:superstab}

At this point \cite[Thm.\ 5.2]{2021arXiv210613261H} is enough to conclude that the topometric space given in \cref{defn:lattice-space} is actually the type space of a stable theory,\footnote{Provided that we verify that the metric on $\AND$ is open, which is straightforward enough.} but given the special form of the type space involved, we can do better.

There is a common pattern among the example given here and many of the examples constructed in \cite[Sec.\ C.1]{HansonThesis}, which is that theory corresponding to the type space in question has the type space \emph{itself} as a model in the following sense.

\begin{defn}\label{defn:assoc-struct}
  For any compact topometric space $(X,\tau,d)$, we write $\Lc_X$ for the metric language containing a predicate symbol $P_f$ for each continuous $f : X \to [0,1]$, where the modulus of uniform continuity $\alpha_{P_f}$ of $P_f$ is chosen so that $f$ is $\alpha_{P_f}$-uniformly continuous on $X$. (Furthermore, if $f$ is Lipschitz and $r$ is the optimal Lipschitz constant for $f$, we take $\alpha_{P_f}(x)$ to be $rx$.)

  We write $M_X$ for the $\Lc_X$-structure whose underlying metric space is $(X,d)$ and in which $P_f^{M_x}(a) = f(a)$ for all $P_f \in \Lc_X$ and $a \in M$. We write $T_X$ for $\Th(M_x)$.
\end{defn}

It follow from \cite[Lem.\ 1.15, Prop.\ 1.17]{BenYaacov2008} that any continuous function from a compact topometric space to $\Rb$ is automatically uniformly continuous (see also \cite[Prop.\ 2.1.2(v)]{HansonThesis} for a direct proof of the relevant special case). Therefore $\Lc_X$ is always well defined.

For any compact topometric space $X$, we have a natural projection map $\pi_X : S_1(T_X) \to X$. (This follows from the fact that a point $x$ in $X$ is uniquely determined by the its quantifier-free $1$-type in $M_X$.) For most $X$, this will fail to be a homeomorphism, but in some special circumstances it is. % 

\begin{defn}
  We say that a compact topometric space $X$ is \emph{autological} if $\pi_X : S_1(T) \to X$ is an isometric homeomorphism (i.e., an isomorphism of topometric spaces).
\end{defn}

Although we find autologicality quite amusing, it seems unlikely that it plays any broad role.

To complete our result, we will show that for any finite non-trivial lattice $L$, $\X(L)$ is autological and $T_{\X(L)}$ is weakly minimal with trivial geometry (implying that it is superstable).

\begin{prop}
  A compact topometric space $X$ is autological if and only if every type in $S_1(T_X)$ is realized in $M_X$.
\end{prop}
\begin{proof}
  If $X$ is autological, then for any $p \in S_1(T_X)$, there is an $x \in X$ such that $\tp(x) = p$, which is precisely the required statement.

  Conversely, suppose that every type in $S_1(T_X)$ is realized in $M_X$. Since $\pi_X(\tp(x)) = x$, every type in $S_1(T_X)$ is realized by at most one element of $M_X$. Since every type is realized, they must all be realized by precisely one element of $M_X$. Therefore $\pi_X: S_1(T_X) \to X$ is a bijection. Since $S_1(T_X)$ and $X$ are compact Hausdorff spaces, this implies that $\pi_X$ is a homeomorphism.

  Finally, we need to establish that $\pi_X$ is isometric. We clearly have that for any $p,q \in S_1(T)$, $d(p,q) \leq d(\pi_X(p),\pi_X(q))$ (since $\tp(\pi_X(r)) = r$ for all $r \in S_1(T)$). On the other hand by \cite[Thm.\ 1.6]{BYTopo2010}, we have for any positive $r < d(\pi_X(p),\pi_X(q))$ that there is a $\frac{1}{r}$-Lipschitz function $f: X \to [0,1]$ such that $f(\pi_X(p)) = 0$ and $f(\pi_X(q)) = 1$. Therefore $P_f^p = 0$ and $P_f^q = 1$, implying that $d(p,q) > r$. Since we can do this for any positive $r < d(\pi_X(p),\pi_X(q))$, we have that $d(p,q) \geq d(\pi_X(p),\pi_X(q))$. Therefore $d(p,q) = d(\pi_X(p),\pi_X(q))$, as required.
\end{proof}

\begin{prop}\label{prop:XL-auto}
  For any finite non-trivial lattice $L$, $\X(L)$ is autological.
\end{prop}
\begin{proof}
  Find an ultrafilter $\Uc$ (on some index set $I$) such that the ultrapower $M_{\X(L)}^\Uc$ is $|\Lc_{\X(L)}|^+$-saturated. In particular, $M_{\X(L)}^\Uc$ realizes all types in $S_1(T_{\X(L)})$. Fix $a \in M_{\X(L)}^\Uc$. We need to argue that $a \equiv \pi_{\X(L)}(\tp(a))$.

  There are three kinds of points in $\X(L)$: points $x$ for which $d(x,y) = 1$ for all $y \neq x$, points $x$ for which there is precisely one $y$ for which $0<d(x,y) < 1$, and points $x$ for which there are precisely two points, $y$ and $z$, such that $0<d(x,y) < 1$ and $d(x,y)=d(x,y)=d(y,z)$. 

  The set of points of the first kind is closed and the sets of points of the second two kinds are open. Clearly if $a \in M_{\X(L)}^\Uc$ is a limit of points of the first kind, it will satisfy the same property in $M_{\X(L)}^{\Uc}$. Therefore the map that switches $a$ and $\pi_X(\tp(a))$ is an automorphism of $M_{\X(L)}^{\Uc}$ and we have that $a \equiv \pi_X(\tp(a))$. 

  Suppose that $a = (x_i)_{i\in I}/\Uc$ is some element of $M_{\X(L)}^\Uc$ where $x_i$ is a point of the second kind for a $\Uc$-large set of indices $i \in I$. The family $(x_i)_{i\in I}$ must eventually concentrate in a single copy of $\AND$, and in that copy it will be in the region $\{\langle x,\pm x \rangle : -1 < x < 0\}$. For each $i \in I$ for which $x_i$ is a point of the second kind, let $y_i$ be the unique point in $X$ such that $d(x_i,y_i) < 1$. Let $b = (y_i)_{i\in I}/\Uc$. There are three possibilities. Either $\lim_{i \to \Uc}d(x_i,y_i) = 0$, $\lim_{i \in \Uc}d(x_i,y_i) \in (0,1)$, or $\lim_{i \in \Uc}d(x_i,y_i) = 1$. In the first and third case, we have that $\pi_X(\tp(a))$ is a point of the first kind and once again the map that switches $a$ and $\pi_X(\tp(a))$ is an automorphism of $M^{\Uc}_{\X(L)}$. In the second case, we similarly have that the map that switches $a$ and $\pi_X(\tp(a))$ and switches $b$ and $\pi_X(\tp(b))$ is an automorphism of $M^{\Uc}_{\X(L)}$. Therefore in any case we have that $a \equiv \pi_X(\tp(a))$.

  The argument when $x_i$ is a point of the third kind for a $\Uc$-large set of indices is essentially the same.
\end{proof}

\begin{thm}\label{thm:fin-lat-type-space}
  For any finite non-trivial lattice $L$, there is a weakly minimal theory $T$ with trivial geometry such that the semilattice of definable subsets of $S_1(T)$ is isomorphic to $L$.
\end{thm}
\begin{proof}
  By \cref{prop:XL-lattice-iso}, we know that there is a topometric space $\X(L)$ whose semilattice of definable sets is isomorphic to $L$. By \cref{prop:XL-auto}, we know that $S_1(T_X)$ is isometrically homeomorphic to $X$, so their semilattices of definable sets are isomorphic.

  The proof of \cref{prop:XL-auto} makes it clear that a stable forking relation can be defined on models of $T_X$ by saying that $B \nind_A C$ if and only if there are $b \in B$ and $c \in C$ such that $d(b,A) = 1$, $d(c,A) = 1$, and $d(b,c) < 1$. Furthermore, whenever $d(b,c) < 1$, we have that $b \in \acl(c)$. Therefore the only way for a $1$-type to fork is for it to become algebraic, which implies that the theory is weakly minimal.
\end{proof}

\section{Some infinite semilattices of definable sets}
\label{sec:some-infin-lat}

Using some of the technology from \cref{sec:building-fin-lat}, we are able to realize some particular infinite lattices as the semilattice of definable subsets of a type space. The idea is to build an infinitely large circuit out of copies of $\AND$ and then compactify in an appropriate way (e.g., \cref{fig:omega-p-1}), possibly continuing the construction further (e.g., \cref{fig:omega-p-omega-p-1}). Since we use \cite{2021arXiv210613261H}, the resulting theories are all stable, but superstability is unclear. The issue is that the resulting type spaces are not autological and so we cannot build the corresponding theory in the same way that we did in \cref{thm:fin-lat-type-space}. This naturally leaves a question.

\begin{quest}
  Are the type spaces constructed in Propositions~\ref{prop:ordinal-type-spaces} and \ref{prop:semi-not-lat} and \cref{thm:semilat-defbl-sets} the type spaces of superstable theories? If not, are there superstable theories with the same semilattices of definable sets?
\end{quest}

\subsection{Crisp one-point compactifications}
\label{sec:machinery}

\begin{defn}\label{defn:crisp-one-point}
  Given a topometric space $(X,\tau,d)$, the \emph{crisp one-point compactification of $X$} is the topometric space whose underlying topological space is the one-point compactification of $X$, $X\cup\{\infty\}$, and whose metric is $d$ extended so that $\infty$ is crisply embedded (i.e., $d(\infty,x) = 1$ for all $x \in X$).

  We denote the crisp one-point compactification of $(X,\tau,d)$ by $(X^\ast,\tau^\ast,d^\ast)$ if it exists.
\end{defn}

(Recall that we are taking all metrics to be $[0,1]$-valued.) Note that while the object described in \cref{defn:crisp-one-point} always exists, it can in general fail to be a topometric space. For instance, an infinite discrete space with a discrete $\{0,\frac{1}{2}\}$-valued metric has no crisp one-point compactification. 

\begin{lem}\label{lem:compactification-of-open-is-open}
   Fix a topometric space $(X,\tau,d)$. If $X$ has an open metric and $X^\ast$ exists, then $d^\ast$ is open as well.
\end{lem}
\begin{proof}
  Recall that a subset of the topological one-point compactification $(X^\ast,\tau^\ast)$ is open if and only if it is an open subset of $X$ or is the complement (in $X^\ast = X \cup \{\infty\}$) of a closed compact subset of $X$. Since $X$ is a topometric space, it is automatically Hausdorff and thus all compact subsets of it are closed.

  If $U \subseteq X$ is open, then for any $r \leq 1$, we have that $U^{<r}\subseteq X$ and so $U^{<r}$ is an open set. For $r > 1$, $U^{<r} = X^\ast$ is also an open set.

  If $U \subseteq X^\ast$ is the complement of a compact subset $F$ of $X$, then for any $r \leq 1$, we have that $U^{<r} = (U\cap X)^{<r}\cup\{\infty\}$, $F \setminus U^{<r}$ is therefore the same as $F \setminus (U\cap X)^{<r}$, which is a closed subset of a compact set and therefore compact itself. Thus $U^{<r}$ is the complement of a closed compact set and so is an open subset of $X^\ast$. If $r \geq 1$, then again $U^{<r} = X^\ast$.
\end{proof}

\begin{prop}\label{prop:compactification-definable-set-characterization}
  Fix a topometric space $(X,\tau,d)$ such that $X^\ast$ exists. A closed set $D \subseteq X^\ast$ is definable if and only $D \cap X$ is definable in $X$ and either
  \begin{itemize}
  \item $D$ is compact and $D \subseteq X$ or
  \item for every $r > 0$, $X \setminus (\int_\tau (D\cap X)^{<r})$ is compact. 
  \end{itemize}
\end{prop}
\begin{proof}
  Assume that $D \subseteq X^\ast$ is definable. For any positive $r \leq 1$, we have that $D^{<r}\cap X = (D\cap X)^{<r}$, so $D\cap X \subseteq (\int_{\tau^\ast}D^{<r}) \cap X \subseteq (D\cap X)^{<r}$ and therefore $D \cap X \subseteq \int_\tau (D\cap X)^{<r}$. Therefore $D\cap X$ is definable in $X$.

  If $\infty \notin D$, then $D$ must be compact, since $D$ is a closed subset of $X^\ast$ not containing $\infty$ (and since $X$ is a topometric space and therefore Hausdorff).
  
  If $\infty \in D$, then for every $r > 0$, we have that $\int_{\tau^\ast} D^{<r}$ is an open neighborhood of $\infty$. Therefore, by the definition of the topology on $X^\ast$, $\int_{\tau} (D\cap X)^{<r}$ is co-compact in $X$.

  Now assume that $D\cap X$ is definable.  If the first bullet point holds, then for any positive $r \leq 1$, we have that $D^{<r} \cap X = (D\cap X)^{<r}$, so $D \subseteq \int_\tau D^{<r} = \int_{\tau^\ast}D^{<r}$. Therefore $D$ is definable in $X^\ast$.

  If the second bullet point holds, then for any positive $r \leq 1$ and $x \in D$, we either have that $x \in X$, in which case $x \in \int_\tau(D\cap X)^{<r} \subseteq \int_{\tau^\ast}D^{<r}$, or $x = \infty$, in which case $\{\infty \}\cup \int_\tau (D\cap X)^{<r} \subseteq D^{<r}$ is an open neighborhood of $\infty$. So in either case, $x \in \int_{\tau^\ast} D^{<r}$. Therefore $D \subseteq \int_{\tau^\ast} D^{<r}$ and since we can do this for any sufficiently small $r > 0$, $D$ is definable in $X^\ast$.
\end{proof}

\subsection{Directed systems of topometric spaces}
\label{sec:directed-systems}

We need the following definition and technical lemmas to deal with the fact that direct limits (also called directed colimits) seem to be rather delicate in the category of topometric spaces.

\begin{defn}\label{defn:directed-system}
  Fix a directed set $I$, a family $(X_i,\tau_i,d_i)_{i\in I}$ of topometric spaces, and isometric topological embeddings $f_{ij}:X_i \to X_j$ for each pair $i\leq j \in I$. Suppose that this data forms a directed system. Let $(X_I,\tau_I,d_I)$ be the direct limit of this system (in the sense that $(X_I,\tau_I)= \lim_{i \in I}(X_i,\tau_i)$ and $(X_I,d_I) = \lim_{i \in I}(X_i,d_i)$) and let $f_{jI}:X_j \to X_I$ be the induced inclusion maps.

  We say that the directed system of topometric spaces $((X_i,\tau_i,d_i)_{i \in I},(f_{ij})_{i\leq j \in I})$ is \emph{crisp} if $f_{ij}[X_i]$ is crisply embedded in $X_j$ for any $i \leq j \in I$. We say that it is \emph{eventually open} if for any $x \in X_I$, there is a $j \in I$ such that $x \in \int_{\tau_I}f_{jI}[X_j]$.
\end{defn}

\begin{lem}\label{lem:directed-system-definable-set-characterization}
  Let $((X_i,\tau_i,d_i)_{i \in I},(f_{ij})_{i\leq j \in I})$ be a crisp directed system of topometric spaces satisfying that $\lim_{i \in I}(X_i,\tau_i,d_i) = (X_I,\tau_I,d_I)$ is a topometric space. Fix a closed set $D \subseteq X_I$.
  \begin{enumerate}
  \item If $D$ is definable in $X_I$, then $D \cap X_j$ is definable in $X_j$ for every $j \in I$.
  \item If $d_j$ is open and $D \cap X_j$ is definable in $X_j$ for every $j \in I$, then $D$ is definable in $X_I$.
  \end{enumerate}
\end{lem}
\begin{proof}
  We may assume without loss of generality that each $X_j$ is a subset of $X_I$ and $f_{jI}: X_j \to X_I$ is the identity map.
  
  For 1, fix $r \in (0,1]$ and consider $D^{<r}$. Since each $X_j$ is crisply embedded in $X_I$, we have that $D^{<r}\cap X_j = (D\cap X_j)^{<r}$. Let $U = \int_{\tau_I}D^{<r}$. By assumption, $D \subseteq U$. By the definition of the direct limit topology, $U \cap X_j$ is $\tau_j$-open. Therefore
  \[
    D\cap X_j \subseteq U \cap X_j \subseteq D^{<r}\cap X_j \subseteq (D\cap X_j)^{<r},
  \]
  whence $D \cap X_j \subseteq \int_{\tau_j}(D\cap X_j)^{<r}$. Since we can do this for any sufficiently small $r > 0$, $D\cap X_j$ is definable in $X_j$.

  For 2, it's clear that for any $r > 0$, $D^{<r} = \bigcup_{i \in I}(D\cap X_j)^{<r}$. Furthermore, just as before, we have that if $r \leq 1$, then $(D\cap X_j)^{<r} = D^{<r}\cap X_j$ for every $j \in I$. Therefore, since each $(D\cap X_j)^{<r}$ is $\tau_j$-open, we have that $D^{<r}$ is $\tau_I$-open. For $r \geq 1$, $D^{<r}$ is either $\varnothing$ or $X_I$, so $D^{<r}$ is open in every case, and $D$ is definable in $X_I$.
\end{proof}

\begin{lem}\label{lem:topo-technical}
  Suppose that $((X_i,\tau_i,d_i)_{i \in I},(f_{ij})_{i\leq j \in I})$ is a crisp and eventually open directed system of topometric spaces.
  \begin{enumerate}
  \item  $\lim_{i\in I}(X_i,\tau_i,d_i) = (X_I,\tau_I,d_I)$ is a topometric space.
  \item If $d_i$ is open for every $i \in I$, then $d_I$ is open.
  \item If $X_i$ is compact for every $i \in I$, then $(X_I,\tau_I,d_I)$ has a crisp one-point compactification.
  \end{enumerate}
\end{lem}
\begin{proof}
  Without loss of generality we may identify each $X_i$ with its image $f_{iI}[X_i]\subseteq X_I$, so that the maps $f_{ij}$ and $f_{iI}$ are identity maps. It is immediate that $X_i$ is crisply embedded in $X_I$ for every $i \in I$.

  Recall that a set $U \subseteq X_I$ is $\tau_I$-open if and only if $U\cap X_i$ is $\tau_i$-open for every $i \in I$. Fix $x \in X_I$ and an open neighborhood $U \ni x$. Find $i \in I$ such that $x \in X_i$. Find an $\e > 0$ with $\e < 1$ such that $B_{<\e}(x)\cap X_i \subseteq U\cap X_i$ (which exists since $d_i$ refines the topology $\tau_i$). Since $X_i$ is crisply embedded in $X_I$, we have that $B_{<\e}(x)\cap X_i = B_{<\e}$. Therefore $B_{<\e}(x) \subseteq U$. Since we can do this for any $x$ and $U \ni x$, we have that $d_I$ refines $\tau_I$.

  Fix $x,y \in X_I$ and $r > 0$ such that $d_I(x,y) > r$. We may assume without loss of generality that $r < 1$.  Find $j \in I$ such that $x$ and $y$ are elements of $\int_{\tau_I}X_j$. (We can do this because of the fact that if $x \in \int_{\tau_I}X_j$ and $y \in \int_{\tau_I}x_k$, then for any $\ell \geq j,k$, $\{x,y\} \subseteq \int_{\tau_I}X_{\ell}$.) Since $X_j$ is a topometric space, there are neighborhoods $U \ni x$ and $V \ni y$ in $X_j$ such that for any $x' \in U$ and $y' \in V$, $ d_j(x',y') > r$. 

  We now have that $U \cap \int_{\tau_I}X_j \ni x$ and $V \cap \int_{\tau_I}X_j \ni y$ are neighborhoods in $X_I$ with the same property. Since we can do this for any $x,y \in X_I$, we have that $d_I$ is lower semi-continuous and so 1 holds (i.e., $X_I$ is a topometric space).

  For 2, fix an open set $U$. For each $x \in U$, find a $j(x) \in I$ such that $x \in \int_{\tau_I}X_{j(x)}$ and let $V_x = U \cap \int_{\tau_I}X_{j(x)}$. We clearly have that $U = \bigcup_{x \in U}V_x$. Fix $r > 0$. If $r > 1$, then $U^{<r} = X_I$ is an open set, so assume that $r \leq 1$. We have that $U^{<r} = \bigcup_{x \in U}V_x^{<r}$. Since $X_{j(x)}$ is crisply embedded in $X_I$, we have that $V_x^{<r} \subseteq X_{j(x)}$. Furthermore, since $d_{j(x)}$ is open, $V_x^{<r}$ is $\tau_{j(x)}$-open.

  We now need to argue that $V_x^{<r}$ is open. Fix $y \in V_x^{<r}$. Find $k \in I$ such that $k \geq j(x)$ and $y \in \int_{\tau_I}X_k$. We now have that $X_j$ is crisply embedded in $X_k$, so $V_x^{<r}$ computed in $X_k$ is the same set as $V_x^{<r}$ computed in $X_j$. Since $d_j$ is open, we have that $V_x^{<r}$ is open as a subset of $X_j$. So now we have that $V_x^{<r} \cap \int_{\tau_I}X_k$ is an open neighborhood of $x$ in $X_k$ (note that $\int_{\tau_I}X_k$ is open in $X_k$). Therefore it is an open neighborhood in $X_I$ as well. Since we can do this for any $y \in V_x^{<r}$, we have that $V_x^{<r}$ is open. Finally, since $V_x^{<r}$ is open for any $x \in U$, $U^{<r}$ is open as well.

  For 3, Since $X_I$ is a topometric space, it is Hausdorff and has that all of its compact sets are closed. As discussed after \cref{defn:crisp-one-point}, it is immediate that $d_I^\ast$ refines $\tau_I^\ast$, so all we need to verify is that $d_I^\ast$ is lower semi-continuous. If $x,y \in X_I$ and $d(x,y) > r > 0$, then since $X_I$ is a topometric space, there are open sets $U \ni x$ and $V \ni y$ such that $\inf\{d(x',y') : x' \in U,~y'\in V\} > r$. The only other case to check is $x \in X_I$ and $\infty$. Find $j \in I$ such that $x \in \int_{\tau_I}X_j$. Since $X_j$ is crisply embedded, we have that for any $x' \in \int_{\tau_I}X_j$ and any $y' \in X^\ast_I \setminus X_j$, $d(x',y') = 1$. Note that since $X_j$ is compact and $X_I$ is Hausdorff, $X_j$ is a closed compact set and so $X^\ast_I\setminus X_j$ is $\tau_I^\ast$-open. Since we can do this for any $x \in X_I$, we have that $d(x,y)$ is lower semi-continuous and hence $X_I$ has a crisp one-point compactification.
\end{proof}

\cref{lem:topo-technical} is far from optimal, but it is unclear how far the techniques in this section can go, so we have not put much effort into sharpening it.

We now have the tools we need to build certain special lattices of definable sets.

\subsection{Successor ordinals}
\label{sec:succ-ord}

Here we will build type spaces in which the semilattices of definable sets correspond to arbitrary successor ordinals. % 

\begin{prop}\label{prop:ordinal-type-spaces}
  For any ordinal $\alpha$, there is a stable theory $T$ in a language of cardinality $\aleph_0 + |\alpha|$ such that the semilattice of definable subsets of $S_1(T)$ is isomorphic to $\alpha+1$. Furthermore, the same is true of the reverse order $(\alpha+1)^\ast$.
\end{prop}

\begin{proof}
  Let $\AND^\dagger$ be $\AND$ with its input vertices soldered together. Refer to the soldered point as $g \in \AND^\dagger$. Let $E$ be the unique non-empty definable proper subset of $\AND^\dagger$ (i.e., the set corresponding to $\{\langle x,y \rangle : y \neq 0\}\cup \{\langle 0,0 \rangle\}$ in $\AND$).

  Let $(X_0,\tau_0,d_0)$ be the one-point topometric space, and let $x_0$ be the unique element of $X_0$. Let $f_{00}:X_0\to X_0$ be the identity map. Note that $d_0$ is open.

  For each successor ordinal $\alpha+1$, given a point $x_\alpha \in X_\alpha$ and the directed system of compact topometric spaces $(X_\beta,\tau_\beta,d_\beta)_{\beta\leq \alpha}$ (and $(f_{\beta\gamma})_{\gamma\leq \gamma \leq \alpha}$ where $X_\beta$ has an open metric for every $\beta \leq \alpha$), let $(X_{\alpha+1},\tau_{\alpha+1},d_{\alpha+1})$ be $X_{\alpha}$ and a new copy of $\AND^\dagger$ soldered together at $x_\alpha \in X_\alpha$ and $g \in \AND^\dagger$. Write $\AND^\dagger_\alpha$ for this copy as a subset of $X_{\alpha+1}$, and write $E_\alpha$ for $\AND^\dagger_\alpha$'s copy of $E$. Let $f_{\alpha,\alpha+1} : X_\alpha \to X_{\alpha+1}$ be the natural inclusion map of $X_\alpha$ into $X_{\alpha+1}$. For each $\beta \leq \alpha$, let $f_{\beta,\alpha+1} : X_\beta \to X_{\alpha+1}$ be $f_{\alpha,\alpha+1}\circ f_{\beta\alpha}$. Note that by \cref{lem:soldering} parts 2 and 4, the metric $d_{\alpha+1}$ is open as well. Furthermore, $X_{\alpha+1}$ is clearly compact.

  For a limit ordinal $\lambda$, given $(X_\beta)_{\beta < \lambda}$ and $(f_{\beta\gamma})_{\beta \leq \gamma < \lambda}$, we need to argue that this is a crisp and eventually open directed system of topometric spaces. Crispness is obvious from \cref{defn:soldering}. Furthermore, we clearly have that for any $\beta\leq \gamma < \lambda$, $f_{\beta\gamma}[X_{\beta+1}\setminus \{x_{\beta+1}\}]$ is an open set containing $X_{\beta}$, so the system is eventually open. Therefore by \cref{lem:topo-technical}, $\lim_{\beta<\lambda}(X_\beta,\tau_\beta,d_\beta)$ is a topometric space with an open metric and a crisp one-point compactification. Let $(X_\lambda,\tau_\lambda,d_\lambda)$ be the crisp one-point compactification. By \cref{lem:compactification-of-open-is-open}, $d_\lambda$ is an open metric. For any $\beta < \lambda$, let $f_{\beta\lambda}: X_\beta \to X_\lambda$ be the natural inclusion map produced by composing the inclusion from $X_\beta$ into $\lim_{\gamma<\lambda}X_\gamma$ and the inclusion from $\lim_{\gamma < \lambda}X_\gamma$ into $X_\lambda$. 

\begin{figure}
  \centering
  \begin{tikzpicture}
    \draw (7,1) node[circle,fill=black,scale=0.6]{};
    \draw (7,-1) node[circle,fill=black,scale=0.6]{};
    \foreach \n in {1,2,...,10}{
      \begin{scope}[shift={({10*(1-sqrt(1/\n))},{1})},scale={10*(sqrt(1/\n) - sqrt(1/(\n+1)))/1.7}]
        \draw[very thick] (0.7,0.3) -- (1,0) -- (1.3,-0.3) -- (0.7,-0.3) -- (1.3,0.3) -- cycle;
        \draw[very thick,line cap=round] (0.7,0.3) .. controls ++(-0.5,0) and ++(0,0) .. (0,0);
        \draw[very thick,line cap=round] (0.7,-0.3) .. controls ++(-0.5,0) and ++(0,0) .. (0,0);
        \draw[very thick] (1,0) -- (1.7,0); % 
    \end{scope}
    \begin{scope}[shift={({10*(1-sqrt(1/(\n+1))},{-1})},scale={-10*(sqrt(1/\n) - sqrt(1/(\n+1)))/1.7}]
      \draw[very thick] (0.7,0.3) -- (1,0) -- (1.3,-0.3) -- (0.7,-0.3) -- (1.3,0.3) -- cycle;
      \draw[very thick,line cap=round] (0.7,0.3) .. controls ++(-0.5,0) and ++(0,0) .. (0,0);
      \draw[very thick,line cap=round] (0.7,-0.3) .. controls ++(-0.5,0) and ++(0,0) .. (0,0);
      \draw[very thick] (1,0) -- (1.7,0); % 
    \end{scope}
    }
  \end{tikzpicture}
  \caption{$\omega + 1$ and $(\omega+1)^\ast$}
  \label{fig:omega-p-1}
\end{figure}
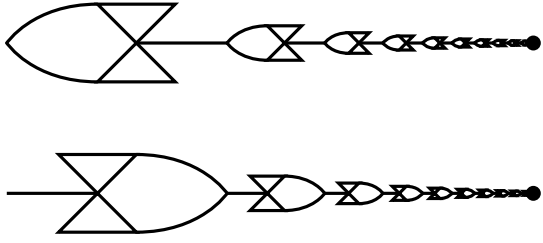

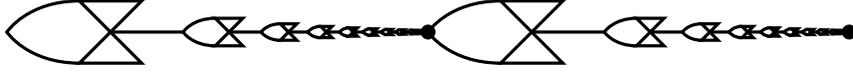
\begin{figure}
  \centering
  \begin{tikzpicture}[scale=0.8]
    \draw (7,1) node[circle,fill=black,scale=0.6]{};
    \foreach \n in {1,2,...,10}{
      \begin{scope}[shift={({10*(1-sqrt(1/\n))},{1})},scale={10*(sqrt(1/\n) - sqrt(1/(\n+1)))/1.7}]
        \draw[very thick] (0.7,0.3) -- (1,0) -- (1.3,-0.3) -- (0.7,-0.3) -- (1.3,0.3) -- cycle;
        \draw[very thick,line cap=round] (0.7,0.3) .. controls ++(-0.5,0) and ++(0,0) .. (0,0);
        \draw[very thick,line cap=round] (0.7,-0.3) .. controls ++(-0.5,0) and ++(0,0) .. (0,0);
        \draw[very thick] (1,0) -- (1.7,0); % 
    \end{scope}
  }
  \begin{scope}[shift={(7,0)}]
    \draw (7,1) node[circle,fill=black,scale=0.6]{};
    \foreach \n in {1,2,...,10}{
      \begin{scope}[shift={({10*(1-sqrt(1/\n))},{1})},scale={10*(sqrt(1/\n) - sqrt(1/(\n+1)))/1.7}]
        \draw[very thick] (0.7,0.3) -- (1,0) -- (1.3,-0.3) -- (0.7,-0.3) -- (1.3,0.3) -- cycle;
        \draw[very thick,line cap=round] (0.7,0.3) .. controls ++(-0.5,0) and ++(0,0) .. (0,0);
        \draw[very thick,line cap=round] (0.7,-0.3) .. controls ++(-0.5,0) and ++(0,0) .. (0,0);
        \draw[very thick] (1,0) -- (1.7,0); % 
      \end{scope}
    }
  \end{scope}

  \end{tikzpicture}
  \caption{$\omega + \omega + 1$}
  \label{fig:omega-p-omega-p-1}
\end{figure}
  
  Now for any ordinal $\alpha$, consider $(X_\alpha,\tau_\alpha,d_\alpha)$. We have by induction that this is a compact topometric space with an open metric. (Note that if a directed system of topometric spaces has a last element, then it is trivially eventually open, so \cref{lem:topo-technical} applies even at successor stages.) We need to argue that the partial order of definable subsets of $X_\alpha$ is order-isomorphic to $\alpha+1$. For $\beta < \alpha$, we will regard $X_\beta$ as a subset of $X_\alpha$. For any $\beta<\alpha$, consider the set $D_\beta \coloneqq X_\beta \cup E_\beta$. For $\gamma \leq \beta$, $D_\beta \cap X_\gamma = X_\gamma$, which is definable in $X_\gamma$. For $\gamma = \beta+1$, $D_\beta \cap X_\gamma$ is definable by \cref{lem:soldering}. We now need to argue that $D_\beta \cap X_\gamma = D_\beta$ is definable in $X_\gamma$ for each $\gamma > \beta+1$ by induction. If $D_\beta \cap X_\gamma$ is definable in $X_\gamma$, then $D_\beta \cap X_{\gamma+1} = D_\beta$ is definable in $X_{\gamma+1}$ by \cref{lem:soldering}. For a limit $\lambda > \beta+1$, we have that $D_\beta$ is definable in $\lim_{\gamma < \lambda}X_\gamma = \bigcup_{\gamma < \lambda}X_\gamma$ by \cref{lem:directed-system-definable-set-characterization}. Therefore $D_\beta$ is definable in $X_\lambda$ by \cref{prop:compactification-definable-set-characterization}. Therefore, by induction, we have that $D_\beta$ is definable in $X_\alpha$. 

  It is clear that if $\beta < \gamma < \alpha$, then $D_\beta \subset D_\gamma$. Furthermore $D_\beta \notin \{\varnothing,X_\alpha\}$ for every $\beta < \alpha$. Therefore the family of definable sets $\{\varnothing\}\cup\{D_\beta : \beta < \alpha\} \cup \{X_\alpha\}$ has order type $1+\alpha+1 = \alpha+1$. 

  Now finally, we just need to argue that these are the only definable subsets of $X_\alpha$. Let $D \subseteq X_\alpha$ be a non-empty definable set that is not $X_\alpha$. Let $\beta < \alpha$ be the smallest such that $\AND^\dagger_\beta \not \subseteq D$. By Lemmas~\ref{lem:soldering} and \ref{lem:directed-system-definable-set-characterization}, it must be the case that either $D \cap \AND^\dagger_\beta$ is empty or $D\cap \AND^\dagger_\beta = E_\beta$. Since $\AND^\dagger_\gamma \subseteq D$ for every $\gamma < \beta$ and since $D$ is closed, it must be the case that $D$ contains $x_\beta \in \AND^\dagger_\beta$, so the first case cannot happen and it must be that $D\cap \AND^\dagger_\beta = E_\beta$. Now, for the sake of contradictions, assume that there is a $\gamma \in (\beta,\alpha)$ such that $D \cap \AND^\dagger_\gamma \neq \varnothing$. Let $\gamma$ be the least such. By Lemmas~\ref{defn:soldering} and \ref{lem:directed-system-definable-set-characterization}, it must be the case that $x_\gamma \in D$. We know that $x_{\beta+1} \notin D$, so it must be the case that $\gamma > \beta+1$. If $\gamma$ is a successor ordinal, then we must have that $D \cap \AND^\dagger_{\gamma - 1} \neq \varnothing$, which is a contradiction. Therefore we must have that $\gamma$ is a limit ordinal. But now $D\cap X_\gamma$ contains $x_\gamma$ (the point at infinity in $X_\gamma$) as an isolated point, which is impossible by \cref{prop:compactification-definable-set-characterization}. Therefore no such $\gamma$ can exist. Therefore $D = D_\beta$. Since we can do this for every non-empty, proper definable subset $D$, we have that the semilattice of definable subsets of $X_\alpha$ is order-isomorphic to $\alpha+1$. The result then follows by \cite[Thm.\ 5.2]{2021arXiv210613261H}. 

  For the reverse order, perform the above construction with the orientation of $\AND^\dagger$ reversed. (See \cref{fig:omega-p-1}.) We will write $X_\alpha^\ast$, $x_\alpha^\ast$, $\AND^{\dagger\ast}_\alpha$, and $E_\alpha^\ast$ for the corresponding objects in this construction. (In particular, note that for any $\alpha$, $x_\alpha^\ast$ is the element of $\AND^{\dagger\ast}_\alpha$ corresponding to the point $\langle 2,0 \rangle$ in $\AND$.) Now, fix an $\alpha$ and for any $\beta < \alpha$, write $D^\ast_{\beta}$ for the set $(X^\ast_\alpha \setminus X^\ast_\beta)\cup E^\ast_\beta$. By essentially the same argument as before, we have that $D^\ast_\beta$ is a definable subset of $\alpha$ for any $\beta < \alpha$. Furthermore each $D^\ast_\beta$ is neither empty nor all of $X_\alpha$ and for any $\beta < \gamma < \alpha$, $D_\beta \supset D_\gamma$, so we have that there is a family of definable sets of order type $1+\alpha^\ast + 1 = (\alpha +1)^\ast$ in $X^\ast_\alpha$.

  Again we need to argue that every definable $D \subseteq X^\ast_\alpha$ is either $\varnothing$, $X^\ast_\alpha$, or $D^\ast_\beta$ for some $\beta < \alpha$. Suppose that $D$ is a non-empty definable proper subset of $X^\ast_\alpha$. Let $\beta < \alpha$ be the smallest such that $D \cap \AND^{\dagger \ast}_\beta \neq \varnothing$. If $\beta$ is a successor ordinal, it is immediate that $D \cap \AND^\ast_\beta = E^\ast_\beta$ (otherwise $D$ would contain $x^\ast_\beta$ and so $D \cap \AND^{\delta\ast}_{\beta -1}$ would be non-empty). Suppose that $\beta$ is a limit ordinal and that $\AND^{\dagger\ast}_\beta \cap D \neq E^\ast_\beta$. It must be the case that $x^\ast_\beta \in D$, but this implies that $D \cap X^\ast_\beta = \{x^\ast_\beta\}$, which is not a definable subset of $X^\ast_\beta$. This contradicts \cref{prop:compactification-definable-set-characterization}. Therefore we must have that $D \cap \AND^{\dagger \ast}_\beta = E^\ast_\beta$. Now suppose for the sake of contradiction that there is a $\gamma \in (\beta,\alpha)$ such that $\AND^{\dagger \ast}_\gamma \not \subseteq D$. Let $\gamma$ be the least such. Since it is the least, we must have that $x^\ast_\gamma \in D$, but by \cref{lem:soldering} and  Propositions~\ref{prop:AND-function} and \ref{prop:compactification-definable-set-characterization}, this implies that $\AND^{\dagger\ast}_\gamma \subseteq D$, which is a contradiction. Therefore no such $\gamma$ can exist and we have that $D = D^\ast_\beta$, as required. Since we can do this for any non-empty definable proper subset $D\subset X^\ast_\alpha$, we have that the semilattice of definable sets in $X^\ast_\alpha$ is order-isomorphic to $(\alpha+1)^\ast$. The result again follows by \cite[Thm.\ 5.2]{2021arXiv210613261H}. 

  To get the cardinality bound on the language of the theory $T$, note that a basic inductive argument shows that for any $\alpha$, $X_\alpha$ has a base of cardinality at most $\aleph_0+|\alpha|$. This implies that there is a reduct $T_0$ of $T$ in a language of cardinality at most $\aleph_0+|\alpha|$ such that $S_1(T_0)$ and $S_1(T)$ are isometrically homeomorphic. The argument for the reverse order case is the same.  
\end{proof}

It is possible to solder the type spaces in \cref{fig:omega-p-1} together in such a way that the resulting semilattice of definable sets is isomorphic to $1+\Zb + 1$. We will not write this out explicitly however, as we prove a more general statement in \cref{sec:comp-lat}.

\subsection{A semilattice that is not a lattice}
\label{sec:semi-not-lat}

In this section we will give an example of a type space in which the semilattice of definable sets is not a lattice, which, while an expected phenomenon, is seemingly a bit hard to come by.

Recall that in a semilattice $L$, an \emph{exact pair} above an ideal $I$ is a pair of elements $a$ and $b$ such that $I = \{x \in L : x \leq a\wedge x \leq b\}$. If $I$ has no largest element, this state of affairs entails that $\{a,b\}$ does not have a greatest lower bound.

Let $X_\omega$ and $x_\omega$ be as in the proof of \cref{prop:ordinal-type-spaces}. Solder three copies of $\AND$ to $x_\omega$ in the configuration seen in \cref{fig:exact-pair} and call the resulting space $Y$. It is straightforward to verify that the lattice of definable subsets of $(Y\setminus X_\omega)\cup \{x_\omega\}$ is isomorphic to the two-element Boolean algebra. Let $A$ and $B$ be the two non-empty proper definable subsets of $(Y\setminus X_\omega)\cup \{x_\omega\}$.

\begin{prop}\label{prop:semi-not-lat}
  The non-empty definable proper subsets of $Y$ are precisely $D_\alpha$ for $\alpha < \omega$ (as defined in the proof of \cref{prop:ordinal-type-spaces}), $X_\omega\cup A$, and $X_\omega \cup B$. In particular, $X_\omega \cup A$ and $X_\omega \cup B$ form an exact pair above the ideal $\{\varnothing\}\cup\{D_\alpha : \alpha < \omega\}$ and so have no meet. 
\end{prop}
\begin{proof}
  By the same reasoning as in the proof of \cref{prop:ordinal-type-spaces}, we can establish that $D_\alpha$ is definable in $Y$ for any $\alpha < \omega$. Furthermore, it follows from \cref{lem:soldering} that $X_\omega \cup A$ and $X_\omega \cup B$ are definable.

  It follows from the proof of \cref{prop:ordinal-type-spaces} that the only sets $D \subseteq X_\omega$ that are definable in $X_\omega$ are $\varnothing$, $X_\omega$, and the sets $\{D_\alpha: \alpha < \omega\}$. 
  
  It is straightforward to verify that if $D\subseteq (Y\setminus X_\omega)\cup \{x_\omega\}$ is non-empty and definable in $(Y\setminus X_\omega)\cup \{x_\omega\}$, then $x_\omega \in D$.

  Therefore, by \cref{lem:soldering}, we have that the definable subsets of $Y$ are precisely $\varnothing$, $Y$, $\{D_\alpha: \alpha < \omega\}$, $X_\omega \cup A$, and $X_\omega\cup B$, and the result follows.
\end{proof}

We have included this proposition not because it resolves an outstanding question about stable theories in continuous logic,\footnote{Even in the strongly minimal theory of $\Rb$ with the metric $d(x,y) = \min(|x-y|,1)$, the semilattice of definable sets fails to be a lattice. It is possible to concoct two definable sets $D_0$ and $D_1$ in $S_1(\Rb)$ satisfying that $D_0 \cap D_1$ is $\Zb$ together with the strongly minimal type. This is not definable but every finite subset of $\Zb$ is, which establishes that $D_0$ and $D_1$ have no meet.} but more because it shows that the techniques presented here can be used to build semilattices that are not lattices. Note that every other semilattice of definable sets presented in this paper is actually a complete lattice.

\begin{figure}
  \centering
  \begin{tikzpicture}
      \draw (7,0) node[circle,fill=black,scale=0.55]{};
      \foreach \n in {1,2,...,10}{
        \begin{scope}[shift={({10*(1-sqrt(1/\n))},{0})},scale={10*(sqrt(1/\n) - sqrt(1/(\n+1)))/1.7}]
          \draw[very thick] (0.7,0.3) -- (1,0) -- (1.3,-0.3) -- (0.7,-0.3) -- (1.3,0.3) -- cycle;
          \draw[very thick,line cap=round] (0.7,0.3) .. controls ++(-0.5,0) and ++(0,0) .. (0,0);
          \draw[very thick,line cap=round] (0.7,-0.3) .. controls ++(-0.5,0) and ++(0,0) .. (0,0);
          \draw[very thick] (1,0) -- (1.7,0); % 
        \end{scope}
      }
      \foreach \s in {-0.7,0.7}{
        \begin{scope}[shift={(7,0)}]
          \begin{scope}[xscale=0.7,yscale=\s,rotate=-90]
            \coordinate (root) at (0,0);
            \coordinate (LRN) at (1.75,1.25); % 
            \coordinate (URN) at (1.75,3.25); % 
            \coordinate (LN) at (-2,2.25); % 
            
            \draw[draw=none] (LN) -- (0,0) coordinate[pos=0.5] (LL) -- (LRN) coordinate[pos=0.6] (LR) -- (URN) coordinate[pos=0.55] (UR);

            \coordinate (mid) at (0,1.9);
            \draw[very thick] (mid) -- (0,0);
            
            \begin{scope}[shift={(LL)},rotate=135,scale=1.1]
              \draw[very thick] (-0.3,0.3) -- (0,0) -- (0.3,-0.3) -- (-0.3,-0.3) -- (0.3,0.3) -- cycle;
              \draw[very thick] (-0.3,0.3) .. controls ++(-0.3,0) and ($ (root) + (0.2,-0.1) $) .. (root);
              \draw[very thick] (-0.3,-0.3) .. controls ++(-0.5,0) and ($ (root) + (0.11,0) $) .. (root);
            \end{scope}
            
            \begin{scope}[shift={(mid)},rotate=-90,xscale=1.1,yscale=0.9]
              \draw[very thick] (-0.3,0.3) -- (0,0) -- (0.3,-0.3) -- (-0.3,-0.3) -- (0.3,0.3) -- cycle;
              \draw[very thick] (-0.3,-0.3) .. controls ++(-0.6,-0.6) and ($ (LN) + (0,-0.1) $) .. (LL);
            \end{scope}
          \end{scope}
        \end{scope}
      }
  \end{tikzpicture}
  \caption{An exact pair above $\omega$}
  \label{fig:exact-pair}
\end{figure}
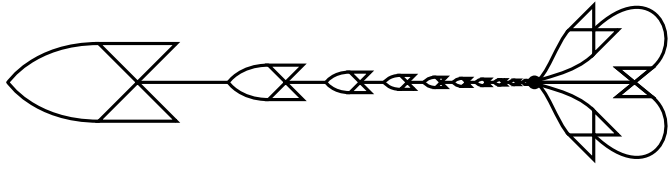

\subsection{Lattices of filters on countable meet-semilattices}% 
\label{sec:comp-lat}

Here we will show that for any countable meet-semilattice $(L,\wedge)$, there is a type space $S_1(T)$ whose join-semilattice of definable sets is isomorphic to the lattice of filters on $L$ (i.e., upwards-closed sets closed under meets).\footnote{In fact, it will be the case that arbitrary unions of definable subsets of $S_1(T)$ are themselves definable, which is not a direct consequence of the semilattice of definable sets being complete. Consider, for example, a discrete theory $T$ with the property that $S_1(T)$ is extremally disconnected. The Boolean algebra of clopen subsets of $S_1(T)$ is complete, but not all unions of clopen sets are clopen, because not all open sets are closed.} This is in some sense an extension of \cref{thm:fin-lat-type-space} (and in particular \cref{lem:lattice-circuit} part 2), but as with the rest of the results in this section, we are only able to build a stable theory.

The lattices of filters in countable meet-semilattices are the same thing as the complete lattices with countable meet-dense subsets.\footnote{A set $A \subseteq L$ is \emph{meet-dense} if every element of $L$ is the meet of some (possibly infinite or empty) subset of $A$.} This includes any countable complete lattice, such as $1+\Zb  +1$ and the Rieger--Nishimura lattice (i.e., the free Heyting algebra over one generator), and many partial orders familiar from analysis, such as $([0,1],\leq)$ and the Boolean algebra of measurable subsets of $[0,1]$ modulo Lebesgue measure $0$.

There are five steps to the construction.
\begin{enumerate}
\item We will build a non-compact, locally compact topometric space $\Y_0(L)$ (with an open metric) whose semilattice of definable sets is the required lattice. This space will be locally compact and have a continuous function $\nu$ to $\Rb_{\geq 0}$ with the property that for every $r > 0$, $\nu^{-1}[[0,r]]$ is compact. Every element $a$ of $L$ will correspond to an open subset $U_a$ of $\Y_0(L)$, which will satisfy $a < b \To U_a \supset U_b$. This will be conceptually similar to the construction in \cref{sec:building-fin-lat}, but we will need to stretch out the nodes corresponding to each element to avoid soldering infinitely many wires to a single point. A filter $F$ will map to the definable set $\Y_0(L) \setminus \bigcap_{a \in F}U_a$.
\item We add two additional copies $\Y_1(L)$ and $\Y_2(L)$ of $\Y_0(L)$ which take turns getting arbitrarily close to $\Y_0(L)$ in the limit as $\nu(x) \to \infty$. While one copy is close to $\Y_0(L)$, the other will retreat to a distance of $1$, so that we can safely solder points together in it without spoiling the openness of the metric.
\item We periodically solder points together in each of the two extra copies of $\Y_0(L)$ (in positions where they have retreated to a distance of $1$) in order to `short circuit' the behavior of definable sets in $\Y_1(L)$ and $\Y_2(L)$. In particular, for any definable set $D$ and either $i \in \{1,2\}$, either $\Y_i(L) \subseteq D$ or $\Y_i(L) \cap D = \varnothing$. 
\item We solder points in $\Y_1(L)$ and $\Y_2(L)$ to some point in $\Y_0(L)$ corresponding to $L$'s bottom, $0^L$. This will ensure that any non-empty definable subset of the space contains all of $\Y_1(L)$ and $\Y_2(L)$. 
\item We take the crisp one-point compactification of the space, adding the point $\infty$. Any non-empty definable set will necessarily contain $\infty$. We then argue that the semilattice of definable sets is unchanged.
\end{enumerate}

In particular, we should note that this does not give us a general method of embedding a non-compact, locally compact topometric space into a compact topometric space while preserving the semilattice of definable sets, as our method will rely heavily on the special form of $\Y_0(L)$.

\subsubsection{Step 1: The space $\Y_0(L)$}
\label{sec:Y-0-L}

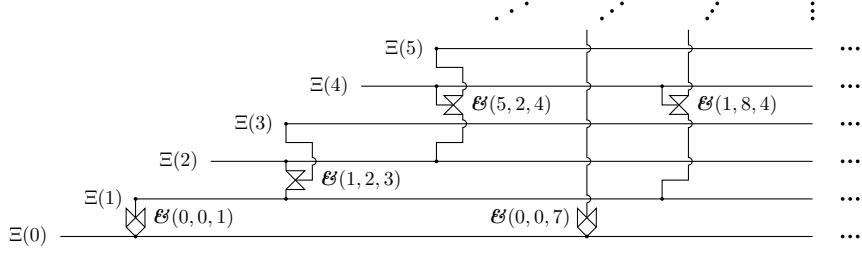
\begin{figure}
  \centering
  \begin{tikzpicture}[yscale=0.5]
    \foreach \i in {0,...,5}{
      \draw (\i,\i) node[left,scale=0.75,shift={(-0.1,0)}]{$\Xi(\i)$} -- (10,\i);
    }

    \foreach \i in {0,1,2,3}{
      \foreach \j in {-0.2,0,0.2}{
        \draw ({6+(\i*4/3)+\j*(1-(\i/3))},{6 + \j}) node[circle,fill=black,scale=0.15]{};
      }
    }

    \foreach \i in {0,...,5}{
      \foreach \j in {-0.2,0,0.2}{
        \draw ({10.5+(\j*0.5)},{\i}) node[circle,fill=black,scale=0.15]{};
      }
    }

    \begin{scope}[yscale=-1,shift={(0,-1)}]
    \draw (1,1) coordinate (s) node[circle,fill=black,scale=0.15]{} -- ++(-0.125,-0.25) -- ++(0,-0.5) -- ++(0.25,0.5) coordinate[pos=0.5] (d) coordinate[pos=1] (e) -- (s);% 
    \draw (e) -- ++(0,-0.5) -- ++(-0.25,0.5);
    \draw (1,0) node[circle,fill=black,scale=0.15]{} -- (d);
    \draw (1.125,0.5) node[right,scale=0.75]{$\AND(0,0,1)$};
  \end{scope}

  \begin{scope}[yscale=-1,shift={(6,-1)}]
    \draw (1,1) coordinate (s) node[circle,fill=black,scale=0.15]{} -- ++(-0.125,-0.25) -- ++(0,-0.5) -- ++(0.25,0.5) coordinate[pos=0.5] (d) coordinate[pos=1] (e) -- (s);% 
    \draw (e) -- ++(0,-0.5) -- ++(-0.25,0.5);
    \draw (d) -- ++(0,-0.4);
    \foreach \i in {0,...,3}{
      \draw ({1},{0.1-\i}) arc[start angle = 90,delta angle = -180,x radius = 0.05,y radius = 0.1] -- ++(0,-0.8);
    }
    \draw (1,-3.9) arc[start angle = 90,delta angle = -180,x radius = 0.05,y radius = 0.1] -- ++(0,-0.4);
    \draw ({1-0.125},{0.5}) node[left,scale=0.75]{$\AND(0,0,7)$};
  \end{scope}

    \draw (3,3) node[circle,fill=black,scale=0.15]{} -- (3,2.5) -- (3.35,2.5) -- ++(0,-0.4) arc[start angle = 90,delta angle = -180,x radius = 0.05,y radius = 0.1]  -- ++(0,-0.4) node[right,scale=0.75]{$\AND(1,2,3)$} -- (3.125,1.5);
    \draw (3,1) node[circle,fill=black,scale=0.15]{} -- (3,1.25) -- (3.25,1.75) -- (3,1.75);
    \draw (3,2) node[circle,fill=black,scale=0.15]{} -- (3,1.75) -- (3.25,1.25) -- (3,1.25);

    \draw (5,5) node[circle,fill=black,scale=0.15]{} -- (5,4.5) -- ++(0.35,0) -- ++(0,-0.4) arc[start angle = 90,delta angle = -180,x radius = 0.05,y radius = 0.1] -- ++(0,-0.15) coordinate (a) -- ++(-0.25,-0.5) coordinate[pos=0.5] (c) -- ++(0.25,0) coordinate (b);
    \draw[draw=none] (a) -- (b) node[midway,right,scale=0.75]{$\AND(5,2,4)$};
    \draw (5,2) node[circle,fill=black,scale=0.15]{} -- ++(0,0.5) -- ++(0.35,0) -- ++(0,0.4) arc[start angle = -90,delta angle = 180,x radius = 0.05,y radius = 0.1] -- (b) -- ++(-0.25,0.5) -- (a);
    \draw (c) -- (5,3.5) -- (5,4) node[circle,fill=black,scale=0.15]{};

    \draw (8.35,5.5) -- ++(0,-0.4) arc[start angle = 90,delta angle = -180,x radius = 0.05,y radius = 0.1] -- ++(0,-0.8) arc[start angle = 90,delta angle = -180,x radius = 0.05,y radius = 0.1] -- ++(0,-0.15) coordinate (aa) -- ++(-0.25,-0.5) coordinate[pos=0.5] (cc) -- ++(0.25,0) coordinate (bb);
    \draw[draw=none] (aa) -- (bb) node[midway,right,scale=0.75]{$\AND(1,8,4)$};
    \draw (8,1) node[circle,fill=black,scale=0.15]{} -- ++(0,0.5) -- ++(0.35,0) -- ++(0,0.4) arc[start angle = -90,delta angle = 180,x radius = 0.05,y radius = 0.1] -- ++(0,0.8) arc[start angle = -90,delta angle = 180,x radius = 0.05,y radius = 0.1] -- (bb) -- ++(-0.25,0.5) -- (aa);
    \draw (cc) -- (8,3.5) -- (8,4) node[circle,fill=black,scale=0.15]{};

  \end{tikzpicture}
  \caption{$\Y_0(L)$ with some copies of $\AND$ shown}
  \label{fig:Y-0-L}
\end{figure}

\begin{defn}
  We write $\Xi$ for the set $\{\langle x,y \rangle \in \Rb \times \Nb : x \geq y\}$, which we regard as a topometric space with the induced topology and a discrete metric.

  Let $(L,\wedge)$ be a countable meet-semilattice with a given enumeration $(\ell_n)_{n<\omega}$. We will always assume that $\ell_0 = 0^L$ (i.e., the bottom element of $L$). We write $\Y_0(L)$ for the topometric space consisting of $\Xi$ together with, for each triple $\langle a,b,c\rangle \in \Nb^3$ satisfying $\ell_a \wedge \ell_b \leq \ell_c$, a copy of $\AND$ with the input vertices soldered to $\langle \max(a,b,c),a \rangle$ and $\langle \max(a,b,c),b \rangle$ and the output vertex soldered to $\langle \max(a,b,c),c \rangle$.

  We write $\Xi(a)$ for the set $\{\langle x,a \rangle : x \geq a\}$, and we write $\AND(a,b,c)$ for the copy of $\AND$ in $\Y_0(L)$ corresponding to $(a,b,c)$, provided that it exists. (See \cref{fig:Y-0-L}.)

  For any filter $F \subseteq L$, we write $\Db_0(F)$ for the subset of $\Y_0(L)$ that is the union of
  \begin{itemize}
  \item $\Xi(a)$ for each $a$ with $\ell_a \notin F$ and
  \item for each triple $\langle a_0,a_1,a_2 \rangle$ with $\ell_{a_0} \wedge \ell_{a_1} \leq \ell_{a_2}$, the unique subset of $\AND(a_0,a_1,a_2)$ that is definable in $\AND(a_0,a_1,a_2)$ and that, for each $i<3$, contains the vertex soldered to $\Xi(a_i)$ if and only if $\ell_{a_i} \notin F$.
  \end{itemize}
  
  We write $\nu_L$ for the function from $\Y_0(L)$ to $\Rb_{\geq 0}$ that takes each element $\langle x,y \rangle$ of $\Xi$ to $x$ and each element of each $\AND(a,b,c)$ to $\max(a,b,c)$.
\end{defn}

\begin{defn}\label{defn:crisp-slicing}
  Given a topometric space $X$, a \emph{crisp slicing of $X$} is a continuous function $\nu: X \to \Rb_{\geq 0}$ such that for any $x,y \in X$, if $\nu(x) \neq \nu(y)$, then $d(x,y) = 1$.
\end{defn}

\begin{lem}\label{lem:Y-0-basic}
  Fix a countable meet-semilattice $L$.
  \begin{enumerate}
  \item $\Y_0(L)$ is well defined and is a topometric space with an open metric.
  \item $\Y_0(L)$ is locally compact. In particular, $\nu^{-1}_L[[0,r]]$ is compact for any $r \geq 0$.
  \item For any filter $F \subseteq L$, $\Db_0(F)$ is well defined and closed.
  \item The function $\nu_L:\Y_0(L) \to \Rb_{\geq 0}$ is continuous.
  \item $\nu_L$ is a crisp slicing of $\Y_0(L)$. % 
  \end{enumerate}
\end{lem}
\begin{proof}
  1 follows from Lemmas~\ref{defn:soldering} and \ref{lem:topo-technical}, where we think of $\Y_0(L)$ as the direct limit of the crisp directed system $(\bigcup_{n<k}\Xi(\ell_n)\cup \bigcup\{\AND(a,b,c) : a,b,c<k,~\ell_a\wedge \ell_b \leq \ell_c\})_{k<\omega}$ with the natural inclusion maps.

  2 follows from the fact that for any $r \geq 0$ only finitely many $\Xi(a)$'s and copies of $\AND$ have points $x$ with $\nu(x) \leq r$.

  For 3, we first need to verify that for each $\AND(a,b,c)$, the prescribed definable-in-$\AND(a,b,c)$ set actually exists. Fix $\langle a,b,c \rangle$ with $\ell_a \wedge \ell_b \leq \ell_c$. By \cref{prop:AND-function}, the only restriction on definable subsets of $\AND(a,b,c)$ is that if they do not contain the vertices corresponding to $a$ and $b$, then they must not contain the vertex corresponding to $c$. So suppose that $\ell_a \in F$ and $\ell_b \in F$ (so that the vertices corresponding to $a$ and $b$ in $\AND(a,b,c)$ need to be not contained in $\Db_0(F)$). Then, since $F$ is a filter, $\ell_a \wedge \ell_b \in F$ and so $\ell_c \in F$ as well. Therefore the vertex corresponding to $c$ in $\AND(a,b,c)$ is not contained in $\Db_0(F)$, and the required definable set exists. The resulting set is closed since it is a locally finite union of closed sets.

  4 and 5 are immediate.
\end{proof}

\begin{prop}\label{prop:Y-0-L-defbl-set-char}
  For any countable meet-semilattice $L$ and filter $F \subseteq L$, $\Db_0(F)$ is a definable subset of $\Y_0(L)$. Furthermore, the map $F \mapsto \Db_0(F)$ is a complete lattice isomorphism from the lattice of filters in $L$ (with join taken to be intersection) to the join-semilattice of definable subsets of $\Y_0(L)$.

  Furthermore, the join of any collection of definable subsets of $\Y_0(L)$ is its set-theoretic union.
\end{prop}
\begin{proof}
  Fix a filter $F$ in $L$. Let $(Y_k)_{k<\omega}$ be the directed system described in the proof of \cref{lem:Y-0-basic}. Since this system is crisp 
  \cref{lem:Y-0-basic}, we just need to verify that $\Db_0(F)\cap Y_k$ is definable for each $k < \omega$, but this is immediate from \cref{lem:soldering} and the fact that $\Xi(a)$ is clopen in $\Xi$.

  By definition, it is clear that $F\mapsto \Db_0(F)$ is order preserving and injective, so to establish that it is a complete lattice isomorphisms, we just need to verify that it is surjective.

  Fix a definable set $D \subseteq \Y_0(L)$. Each $\Xi(a)$ has a connected open neighborhood $U$ such that each $x \in U$ is crisply embedded. Therefore, by \cref{lem:wires}, we have that for each $\ell_a \in L$, either $\Xi(a) \subseteq D$ or $\Xi(a) \cap D = \varnothing$. Let $F(D) = \{\ell_a \in L : \Xi(a) \cap D = \varnothing\}$. We just need to argue that $F(D)$ is a filter and $D = \Db_0(F(D))$. To see that $F(D)$ is a filter, suppose that $\ell_a$ and $\ell_b$ are in $F(D)$. For any $\ell_c \in L$ with $\ell_a \wedge \ell_b \leq \ell_c$, we must have that $D\cap \AND(a,b,c)$ contains neither of its input vertices. Therefore by \cref{lem:soldering} and \cref{prop:AND-function}, $D \cap \AND(a,b,c) = \varnothing$, whereby $D \cap \Xi(c) = \varnothing$ and $\ell_c \in F(D)$. Since we can do this for any $\ell_a,\ell_b\in L$, we have that $F(D)$ is a filter. It is also easy to see that $\Db_0(F(D)) = D$.

  The `Furthermore' statement follows from the fact that clearly for any family $(F_i)_{i\in I}$ of filters in $L$, $\Db_0\left(\bigcap_{i \in I}F_i\right) = \bigcup_{i \in I}\Db_0(F_i)$.
\end{proof}

\subsubsection{Step 2: Taking turns}
\label{sec:taking-turns}

Now we will develop some machinery that we will use in this step of the construction.

\begin{defn}

  Given a topometric space $(X,\tau,d)$ with a crisp slicing $\nu$ and two continuous functions $f_1,f_2 : \Rb_{\geq 0} \to (0,1]$, we write $\Wb(X,\nu,f_1,f_2)$ for the topometric space $X \times 3$ (where $3 = \{0,1,2\}$ has the discrete topology and $X \times 3$ is given the product topology) with the unique metric $d_{\Wb}$ satisfying that
  \begin{itemize}
  \item $d_{\Wb}$ extends $d$ on each copy of $X$,
  \item the function $\nu$ extended to $X\times 3$ by $\nu((x,i)) = \nu(x)$ is a crisp slicing, 
  \item for any $(x,0)$ and $(y,1)$ with $\nu(x)=\nu(y)$, $d_{\Wb}(x,y) = \max(d(x,y),f_1(\nu(x)))$,
  \item for any $(x,0)$ and $(y,2)$ with $\nu(x)=\nu(y)$, $d_{\Wb}(x,y) = \max(d(x,y),f_2(\nu(x)))$, and
  \item for any $(x,1)$ and $(y,2)$ with $\nu(x)=\nu(y)$, $d_{\Wb}(x,y) = \max(d(x,y),\min(f_1(\nu(x))+f_2(\nu(x)),1))$.
  \end{itemize}
\end{defn}

Given any two sets $A$ and $B$ in a metric space $(X,d)$, we'll write $\dinf(A,B)$ for the quantity $\inf\{d(a,b) : a \in A,~b \in B\}$. Recall that the lower semi-continuity condition in the definition of topometric space is equivalent to the following: For any $x,y \in X$ with $d(x,y) > r$, there are open neighborhoods $U \ni x$ and $V \ni y$ such that $\dinf(U,V) > r$.

\begin{lem}
  Let $X$ be a topometric space with a crisp slicing $\nu$ and let $f_1,f_2 : \Rb_{\geq 0} \to (0,1]$ be continuous functions.
  \begin{enumerate}
  \item The metric $d_\Wb$ is well defined.
  \item $\Wb(X,\nu,f_1,f_2)$ is a topometric space.
  \item For both $i \in \{1,2\}$, if $x \in X \times \{i\}$ and $f_i(\nu(x)) = 1$, then $x$ is crisply embedded in $X\times \{i\}$ if and only if it is crisply embedded in $\Wb(X,\nu,f_1,f_2)$.
  \item A closed subset $D \subseteq \Wb(X,\nu,f_1,f_2)$ is definable if and only if $D \cap (X\times \{i\})$ is definable in $X\times \{i\}$ for each $i<3$.
  \item $d_\Wb$ is an open metric if and only if $d$ is an open metric.
  \end{enumerate}
\end{lem}
\begin{proof}
  For 1, since $\nu$ is a crisp slicing and the metric is $[0,1]$-valued, we only need to check the triangle inequality for triples $(x,y,z)$ with $\nu(x)=\nu(y)=\nu(z)$, but this follows easily from the fact that on $\nu^{-1}(r)$, the metric $d_{\Wb}$ is the metric $\max(d,d_3)$, where $d_3$ is the metric on $3$ satisfying $d_3(0,1) = f_1(\nu(x))$, $d_3(0,2) = f_2(\nu(x))$, and $d_3(1,2) = \min(f_1(\nu(x)) + f_2(\nu(x)),1)$. Since $f_1$ and $f_2$ are always positive, this is a metric.

  For 2, the fact that $d_\Wb$ refines the topology is obvious, so we only need to check lower semi-continuity. Fix $(x,i)$ and $(y,j)$ in $\Wb(X,\nu,f_1,f_2)$. If $i = j$, then for any $r < d(x,y)$, we can obviously find open neighborhoods $U \ni x$ and $V \ni y$ with $\dinf(U,V) > r$, since $X$ is a topometric space. So assume that $i \neq j$. If $\nu(x) \neq \nu(y)$, then we can assume that $\nu(x) < s < t < \nu(y)$. We then have $x \in \nu^{-1}[[0,s)]$, $y \in \nu^{-1}[(t,\infty)]$, and $\dinf_\Wb(\nu^{-1}[[0,s)],\nu^{-1}[(t,\infty)]) = 1 > r$, so we have the required neighborhoods. So assume that $\nu(x) = \nu(y)$.

  Let $g(u)$ be $f_1(u)$ if $\{i,j\} = \{0,1\}$, $f_2(u)$ if $\{i,j\} = \{0,2\}$, and $\min(f_1(u)+f_2(u),1)$ if $\{i,j\} = \{1,2\}$. Note that $g(u)$ is always continuous. If $d_\Wb((x,i),(y,j)) = g(\nu(x)) > r$, then by continuity of $\nu$, there will be neighborhoods $U \ni (x,i)$ and $V \ni (y,j)$ with $U\subseteq X \times \{i\}$ and $V \subseteq X \times \{j\}$ such that $\dinf_\Wb(U,V) > r$. If on the other hand, $d_\Wb((x,i),(y,j)) > r \geq g(\nu(x))$, then it must be the case that $d_\Wb((x,i),(y,j)) = d(x,y)$. Find $U,V \subseteq X$ with $x \in U$, $y \in V$, and $\dinf(U,V) > r$. We then clearly have that $\dinf_\Wb(U\times \{i\},V \times \{j\}) > r$, so we are done.

  3 is obvious. 4 follows from the fact that definability is a local property.

  For 5, first note that openness passes to open subspaces, so if $d_\Wb$ is an open metric, then $d$ is an open metric. Now assume that $d$ is an open metric. Fix an open set $U \subseteq \Wb(X,\nu,f_1,f_2)$ and an $r > 0$. Assume that $(x,i) \in U^{<r}$. We need to show that $x \in \int U^{<r}$. This is trivial if $r > 1$, so assume that $r \leq 1$. If there is a $(y,i) \in U$ such that $d(x,y) < r$, then we are done by openness of $d$, so assume that there is no such $(y,i) \in U$. There must be a $(y,j) \in U$ with $j \neq i$ such that $d_\Wb((x,i),(y,j)) < r$. Since $r \leq 1$, this implies that $\nu(x) = \nu(y)$. Let $g(u)$ be defined as before. Note that by definition, we have that $g(\nu(x)) < r\leq 1$. Find an interval $(s,t) \ni \nu(x)$ such that for any $u \in (s,t)$, $g(u) < r$. Now we have that
  \[
    x \in \{(z,i) : (\exists (w,j) \in U)d(z,w) < r\}\cap \nu^{-1}[(s,t)] \subseteq U^{<r}.
  \]
  The set $\{(z,i) : (\exists (w,j) \in U)d(z,w) < r\}$ is open by the openness of $d$, and the set $\nu^{-1}[(s,t)]$ is open by continuity of $\nu$. The fact that their intersection is a subset of $U^{<r}$ is immediate from the definition of $d_\Wb$, so $x \in \int U^{<r}$ as required.
\end{proof}

Fix two continuous functions $h_1,h_2 : \Rb_{\geq 0} \to (0,1]$ satisfying that
\begin{itemize}
\item $h_1(0) = 1$,% 
\item for every odd $n \in \Nb$, $h_1(n) = 1$,% 
\item for every even $n \in \Nb$, $h_2(n) = 1$, and % 
\item for every $x $ with $x \geq 1$, $\min(h_1(x),h_2(x)) = \frac{1}{x}$.
\end{itemize}
It is clear that such functions exist.

\begin{defn}
  We write $\Y_{012}(L)$ for the space $\Wb(\Y_0(L),\nu_L,h_1,h_2)$ (\cref{fig:Y-012-L}). We will identify $\Y_0(L)$ with $\Y_0(L)\times \{0\}$ and we will write $\Y_i(L)$ for $\Y_0(L) \times \{i\}$ for both $i \in \{1,2\}$.
\end{defn}

\begin{figure}
  \centering
  \begin{tikzpicture}
    \draw (0,0) node[left,scale=0.75]{$\Y_0(L)$};
    \foreach \i in {0,...,6}{
      \draw[line width={0.6+0.0*\i}] ({\i},{0}) -- ({\i + 1.5},{0});
    }

    \draw[line width=0.6] (0,1) node[left,scale=0.75]{$\Y_1(L)$} -- (1,1);
    \foreach \i in {1,...,6}{
      \draw[domain=\i:{\i+1.5},variable=\x,line width={0.6+0.0*\i}] plot ({\x},{max(1/\x,sin(90*\x)^20)});
    }
    \draw[line width=0.6] (0,-1) node[left,scale=0.75]{$\Y_2(L)$} -- (1,-1);
    \foreach \i in {1,...,6}{
      \draw[domain=\i:{\i+1.5},variable=\x,line width={0.6+0.0*\i}] plot ({\x},{-max(1/\x,sin(90*(\x+1))^20)});
    }

    \foreach \i in {0}{
      \foreach \j in {-0.2,0,0.2}{
        \draw ({8+(\j*0.5)},{\i}) node[circle,fill=black,scale=0.15]{};
      }
    }

  \end{tikzpicture}
  \caption{$\Y_{012}(L)$}
  \label{fig:Y-012-L}
\end{figure}

\subsubsection{Steps 3-5: Completing the construction}
\label{sec:completing-construction}

\begin{defn}\label{defn:Y-L}
  We write $\Y(L)$ to represent $\Y_{012}(L)$ with the following sets of points soldered together:
  \begin{enumerate}
  \item For each odd $n$, we solder the points $\{(\langle n,m \rangle , 1) : \langle n,m \rangle \in \Xi,~m \leq n\} \subseteq \Y_1(L)$ together (\cref{fig:Y-0}).
  \item For each even $n$, we solder the points $\{(\langle n,m \rangle , 2) : \langle n,m \rangle \in \Xi,~m \leq n\}\subseteq \Y_2(L)$ together.
  \item We solder the points $\{(\langle 0,0 \rangle,i) : i < 3\}$ together. 
  \end{enumerate}
  We write $\pi_{012}$ for the quotient map from $\Y_{012}(L)$ to $\Y(L)$. We also regard $\nu$ as a function on $\Y(L)$, satisfying $\nu(\pi_{012}(x)) = \nu(x)$ for any $x \in \Y_{012}(L)$. 
\end{defn}
Note that the points being soldered are all crisply embedded by the definition of $d_\Wb$ and our choice of $h_1$ and $h_2$. Also note that while we are soldering together infinitely many points, for each $r > 0$, we are only soldering finitely many points in $\nu^{-1}[[0,r]]$. Finally, note that $\nu:\Y(L) \to \Rb_{\geq 0}$ is well defined.

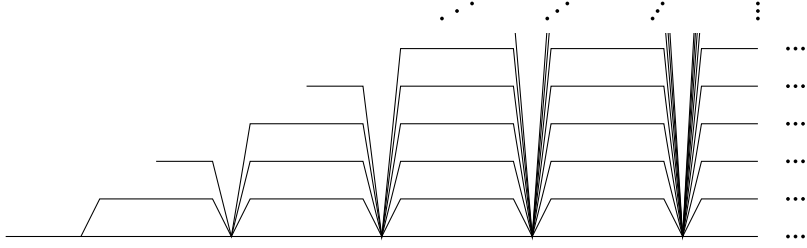
\begin{figure}
  \centering
  \begin{tikzpicture}[yscale=0.5]
    \draw (0,0) -- (10,0);
    \draw (1,0) -- (1.25,1) -- (2.75,1) -- (3,0) -- (3.25,1) -- (4.75 ,1) -- (5,0) -- (5.25,1) -- (6.75,1) -- (7,0) -- (7.25,1) -- (8.75,1) -- (9,0) -- (9.25,1) -- (10,1);
    \draw (2,2) -- (2.75,2) -- (3,0) -- (3.25,2) -- (4.75 ,2) -- (5,0) -- (5.25,2) -- (6.75,2) -- (7,0) -- (7.25,2) -- (8.75,2) -- (9,0) -- (9.25,2) -- (10,2);
    \draw (3,0) -- (3.25,3) -- (4.75 ,3) -- (5,0) -- (5.25,3) -- (6.75,3) -- (7,0) -- (7.25,3) -- (8.75,3) -- (9,0) -- (9.25,3) -- (10,3);
    \draw (4,4) -- (4.75 ,4) -- (5,0) -- (5.25,4) -- (6.75,4) -- (7,0) -- (7.25,4) -- (8.75,4) -- (9,0) -- (9.25,4) -- (10,4);
    \draw (5,0) -- (5.25,5) -- (6.75,5) -- (7,0) -- (7.25,5) -- (8.75,5) -- (9,0) -- (9.25,5) -- (10,5);
    \begin{scope}
      \clip (0,-0.-0.5) rectangle (10,5.4);
      \draw[rounded corners=1] (6.75,6) -- (7,0) -- (7.25,6) -- (8.75,6) -- (9,0) -- (9.25,6);
      \draw[rounded corners=1] (7,0) -- (7.25,7) -- (8.75,7) -- (9,0) -- (9.25,7);
      \draw[rounded corners=1] (8.75,8) -- (9,0) -- (9.25,8);
      \draw[rounded corners=1] (9,0) -- (9.25,9);
    \end{scope}

    \foreach \i in {0,1,2,3}{
      \foreach \j in {-0.2,0,0.2}{
        \draw ({6+(\i*4/3)+\j*(1-(\i/3))},{6 + \j}) node[circle,fill=black,scale=0.15]{};
      }
    }

    \foreach \i in {0,...,5}{
      \foreach \j in {-0.2,0,0.2}{
        \draw ({10.5+(\j*0.5)},{\i}) node[circle,fill=black,scale=0.15]{};
      }
    }

  \end{tikzpicture}
  \caption{The soldering described in \cref{defn:Y-L} part 1}
  \label{fig:Y-0}
\end{figure}

\begin{lem}\label{lem:Y-0-stuff}
  Fix a countable meet-semilattice $L$ with a given enumeration $(\ell_n)_{n<\omega}$ satisfying $\ell_0 = 0^L$.
  \begin{enumerate}
  \item $\Y(L)$ is a well-defined topometric space with an open metric.
  \item A closed set $D \subseteq \Y(L)$ is definable if and only if either $D = \varnothing$ or $D \cap \pi_{012}[\Y_0(L)]$ is definable in $\pi_{012}[\Y_0(L)]$ and $\pi_{012}[\Y_1(L)\cup \Y_2(L)] \subseteq D$.
  \item $\Y(L)$ has a crisp one-point compactification.
  \end{enumerate}
\end{lem}
\begin{proof}
  1 follows from \cref{lem:soldering} and \cref{lem:topo-technical}, where we regard $\Y(L)$ as the limit of the directed system of topometric spaces $(\{x \in \Y(L) : \nu(x) \leq k\})_{k<\omega}$.

 For 2, we have that if $D \subseteq \Y(L)$ is definable, then by Lemmas~\ref{lem:soldering} and \ref{lem:directed-system-definable-set-characterization}, we have that $D \cap \pi_{012}[\Y_0(L)]$ is definable. By \cref{prop:Y-0-L-defbl-set-char}, we have that if $D \cap \pi_{012}[\Y_0(L)]$ is non-empty, then $\pi_{012}[\Xi(0)] \subseteq D$ (because we have assume that $\ell_0 = 0^L$). On the other hand, the same is true of $\Y_1(L)$ and $\Y_2(L)$ (relative to their copies of $\Xi(0)$), but given the points that have been soldered together in these, we have that if a relatively definable set $D \subseteq \pi_{012}[\Y_i(L)]$ (for some $i \in \{1,2\}$) contains $\Y_i(L)$'s copy of $\Xi(0)$, then it must contain all of $\Y_i(L)$'s copy of $\Xi$ and so by \cref{prop:AND-function}, it must contain all of $\Y_i(L)$. Therefore, since we soldered together points in all three copies of $\Xi(0)$, we have that the following are equivalent:
    \begin{itemize}
    \item $D \subseteq \Y(L)$ is non-empty.
    \item $D \cap \pi_{012}[\Y_i(L)] \neq \varnothing$ for some $i < 3$.
    \item $D \supseteq \pi_{012}[\Xi(0)]$.
    \item $D \supseteq \pi_{012}[\Y_i(L)]$ for some $i \in \{1,2\}$.
    \item $D \supseteq \pi_{012}[\Y_i(L)]$ for both $i \in \{1,2\}$.
    \end{itemize}
    Therefore the required statement follows by \cref{lem:soldering}.

  3 follows from \cref{lem:topo-technical}.
\end{proof}

\begin{thm}\label{thm:semilat-defbl-sets}
  For any countable meet-semilattice $L$, there is a stable theory $T$ in a countable language such that the join-semilattice of definable subsets of $S_1(T)$ is isomorphic to the lattice of filters in $L$. Furthermore, for any family $\mathscr{D}$ of definable subsets of $S_1(T)$, the join of $\mathscr{D}$ is the set-theoretic union.
\end{thm}
\begin{proof}
  Fix an enumeration $(\ell_n)_{n<\omega}$ of $L$ with $\ell_0 = 0^L$. Let $\Y^\ast(L)$ be the crisp one-point compactification of $\Y(L)$ (computed with the enumeration $(\ell_n)_{n<\omega}$), which exists by \cref{lem:Y-0-stuff} part 3. Let $\infty$ denote the point at infinity in $\Y^\ast(L)$.

  Let $D \subseteq \Y^\ast(L)$ be a definable set. Since no non-empty definable subsets of $\Y(L)$ are compact (as they must contain all of $\pi_{012}[\Xi(0)]$), we have by \cref{prop:compactification-definable-set-characterization} that if $D$ is non-empty, it must contain $\infty$ and must have that $D \cap \Y(L)$ is definable in $\Y(L)$.

  Now let $D \subseteq \Y(L)$ be a non-empty definable subset of $\Y(L)$. We need to argue that $D \cup \{\infty\}$ is a definable subset of $\Y^\ast(L)$. Clearly $(D\cup \{\infty\})\cap \Y(L)$ is definable in $\Y(L)$, so by \cref{prop:compactification-definable-set-characterization}, we just need to verify that $\Y^\ast(L)\setminus (\int D^{<r})$ is compact for every $r > 0$. By \cref{lem:Y-0-stuff} part 2, $D$ contains all of $\pi_{012}(\Y_1(L)\cup\Y_2(L))$. By the definition of the metric, we have that for any $x \in \Y_{012}(L)$ with $\nu(x) > \frac{1}{r}$, there is a $y \in \Y_1(L) \cup \Y_2(L)$ such that $d(\pi_{012}(x),\pi_{012}(y)) < r$. Therefore, if $x \in \Y(L)$ has $x \notin D^{<r}$, then $\nu(x) \leq \frac{1}{r}$. The set of such $x$'s is compact, so we have that $\Y(L) \setminus (\int D^{<r})$ is compact. Since we can do this for any $r > 0$, we have that $D\cup \{\ast\}$ is definable in $\Y^\ast(L)$ by \cref{prop:compactification-definable-set-characterization}.

  Therefore the semilattices of definable sets in $\Y(L)$ and $\Y^\ast(L)$ are isomorphic. The fact that arbitrary joins are set-theoretic unions follows immediately from \cref{prop:Y-0-L-defbl-set-char}.

  Finally, $\Y^\ast(L)$ is compact and has an open metric by \cref{lem:compactification-of-open-is-open}. Therefore by \cite[Thm.\ 5.2]{2021arXiv210613261H}, there is a stable theory $T$ such that $S_1(T)$ is isometrically homeomorphic to $\Y^\ast(L)$. Since $\Y^\ast(L)$ has a countable topological base, we can find a countable reduct $T'$ of $T$ with the same space of $1$-types if necessary. Therefore we may assume that $T$ has a countable language.
\end{proof}

\bibliographystyle{plain}
\bibliography{../ref}

\begin{thebibliography}{1}

\bibitem{BenYaacov2008}
Ita{\"i} Ben~Yaacov.
\newblock Topometric spaces and perturbations of metric structures.
\newblock {\em Logic and Analysis}, 1(3):235, July 2008.

\bibitem{BYTopo2010}
Ita{\"i} Ben~Yaacov.
\newblock Lipschitz functions on topometric spaces.
\newblock {\em Journal of Logic and Analysis}, 10 2010.

\bibitem{MTFMS}
Ita{\"i} Ben~Yaacov, Alexander Berenstein, C.~Ward Henson, and Alexander
  Usvyatsov.
\newblock {\em Model theory for metric structures}, volume~2 of {\em London
  Mathematical Society Lecture Note Series}, pages 315--427.
\newblock Cambridge University Press, 2008.

\bibitem{HansonThesis}
James Hanson.
\newblock {\em Definability and categoricity in continuous logic}.
\newblock PhD thesis, University of Wisconsin--Madison, 2020.

\bibitem{2021arXiv210613261H}
James {Hanson}.
\newblock {Topometric characterization of type spaces in continuous logic}.
\newblock {\em arXiv e-prints}, page arXiv:2106.13261, June 2021.
\newblock Submitted.

\end{thebibliography}

\end{document}